\documentclass[preprint,10pt]{elsarticle}

\usepackage{syntonly}

\newtheorem{theorem}{Theorem}

\newtheorem{corollary}{Corollary}

\newtheorem{remark}{Remark}
\newenvironment{proof}{{\bf Proof:}}{\hfill$\bull$\medskip}
\newcommand{\bull}{\vrule height 1.8ex width 1.0ex depth 0ex}

\def\XXint#1#2#3{{\setbox0=\hbox{$#1{#2#3}{\int}$ }
\vcenter{\hbox{$#2#3$ }}\kern-.5\wd0}}

\usepackage{graphicx}

\usepackage{threeparttable}
\usepackage{multirow}
\usepackage{amssymb}
\usepackage{latexsym,euscript}
\usepackage{amsmath,amssymb}
\usepackage{amsfonts}
\usepackage{booktabs}
\usepackage{graphicx}
\usepackage{leftidx}

   \def\XXint#1#2#3{{\setbox0=\hbox{$#1{#2#3}{\int}$}
   \vcenter{\hbox{$#2#3$}}\kern-.5\wd0}}

 \def\I{\mathrm{i}}

 \def\I{\mathrm{i}}

\newcommand{\E}{\mathrm{e}}
\newcommand{\D}{\,\mathrm{d}}

\journal{journal}

\begin{document}

\begin{frontmatter}



\title{\textbf{On the numerical quadrature of weakly singular oscillatory integral
and its fast implementation}\tnoteref{label1}}

\author{Zhenhua Xu}
\ead{xuzhenhua19860536@163.com}

\address{College of Mathematics and Information Science, Zhengzhou University of Light Industry, Zhengzhou, Henan 450002, China.}

\begin{abstract}
 In this paper, we present a Clenshaw--Curtis--Filon--type method for the
  weakly singular oscillatory integral with Fourier and Hankel kernels. By interpolating the
  non-oscillatory and nonsingular part of the integrand at $(N+1)$ Clenshaw--Curtis points, the method can
  be implemented in $O(N\log N)$ operations. The method requires the accurate computation of
  modified moments. We first give a method for the derivation of the recurrence relation for the modified moments,
  which can be applied to the derivation of the recurrence relation for the modified moments corresponding to other type
  oscillatory integrals. By using  recurrence relation, special functions and classic quadrature methods, the modified moments
  can be computed accurately and efficiently.
  Then, we present the corresponding error bound in inverse powers of frequencies $k$ and $\omega$
  for the proposed method.  Numerical examples are provided to support the theoretical results and show the efficiency and accuracy of the method.

\end{abstract}


\end{frontmatter}


\section{\textbf{Introduction}}\label{sec:intro}
\hspace{0.1cm}
In this work we consider the evaluation of the weakly singular oscillatory integral
of the form
\begin{eqnarray}\label{eq:a1}
I[f]=\int_0^1 f(x) x^\alpha (1-x)^\beta \E^{\I2kx} H_\nu^{(1)}(\omega x)\D x
\end{eqnarray}
where $\alpha-|\nu|>-1, \beta>-1$, and $k\gg1, \omega\gg1$,   $H^{(1)}_{\nu}( x)=J_{\nu
}(x)+\I Y_{\nu }(x)$ is Hankel function of the first kind of order
$\nu$, and $f$ is a sufficiently smooth function on $[0, 1]$.
In many areas
of science and engineering, for example, in astronomy, optics, quantum mechanics, seismology image processing, electromagnetic
scattering (\cite{Arden}, \cite{Arfken}, \cite{Bao}, \cite{Davis}, \cite{Huybrechs1}), one will come across the computation
of the integral \eqref{eq:a1}.

The integral \eqref{eq:a1} has the following two characteristics:
\begin{enumerate}
\item
When $k+\omega\gg 1$, the integrand becomes highly oscillatory. Consequently, a prohibitively number of quadrature nodes are needed to
obtain satisfied accuracy if one uses classical numerical methods like Simpson rule, Gaussian quadrature, etc.
Moreover, it presents serious
difficulties in obtaining numerical convergence of the integration.

\item The function $H_\nu^{(1)}(x)$ has a  logarithmic singularity for $\nu=0$, and algebraic singularity for $\nu\neq0$ at the point $x$=0.
In addition, if $-1<\alpha, \beta <0$, the integrand also has algebraic singularities at two endpoints, which impacts heavily on its quadrature and its
error bound. For a special case that $\alpha=0, \beta=0$, and $k=0$, the integral can be rewritten in a special form
\begin{equation}\label{eq:a2}
I[f]=\int_0^1 f(x) H_\nu^{(1)}(\omega x)\D x.
\end{equation}
\end{enumerate}

In the last few years, many efficient numerical methods has been devised for the evaluation of oscillatory
integrals. Here, we only mention several main methods,  such as Levin method and Levin-type method \cite{Levin1, Levin2, Olver1}, generalized quadrature rule \cite{Evans1, Evans2},
Filon method and Filon--type method \cite{Dominguez1, Dominguez2, Filon, Iserles, Xiang, Xiang1, Xiang2}, Gauss--Laguerre quadrature \cite{Chen, Chen2, Chen3, Huybrechs1, Huybrechs2, Iserles, Xu1}.
In what follows, we will introduce several other papers related to the integrals considered in this paper.
For the integral $\int_0^1 f(x)x^\alpha (1-x)^\beta \E^{\I kx}\D x$, as early as in 1992, Piessens \cite{Piessens} construct
a fast algorithm to approximate it by truncating $f$ by its Chebyshev series and using the recurrence relation of the
modified moments.
Recently, the references \cite{Kang1, Kang2} developed this method by using a special Hermite interpolation at
Clenshaw--Crutis points and Chebyshev expansion for $\E^{\I kx}$.
If $f$ is analytic  in a sufficiently large
complex region containing $[0,1]$, a numerical steepest
descent method \cite {Kang3} was presented by using complex integration theory. The same idea is also applied
to the the computation of the integral $\int_a^b(x-a)^\alpha(b-x)^\beta\ln(x-a)f(x) \E^{\I\omega x}\D x, \alpha, \beta>-1$, based
on  construction of the Gauss quadrature rule with logarithmic weight function \cite{He}.
For the integral $\int_0^1 f(x)x^\alpha (1-x)^\beta J_m(\omega x)\D x, \alpha, \beta>-1$, a Filon-type method based on
a special Hermite interpolation polynomial at Clenshaw--Curtis points was introduced in \cite{Chen4}. On the
other hand, the reference \cite{Kang4} proposed a Clenshaw--Curtis--Filon method for the computation of the oscillatory Bessel integral
$\int_0^1 f(x)x^\alpha \ln (x)(1-x)^\beta J_m(\omega x)\D x, \alpha, \beta>-1$, with algebraic or logarithmic singularities at the two endpoints.

For the evaluation of the integral \eqref{eq:a1}, the literature \cite{He} transformed it into two line integrals by using the analytic continuation
and the construction of Gauss quadrature rules. However, this method require that $f$ is analytic in a enough large region.
A recent work \cite{Xu2} present a Clenshaw--Curtis--Filon--type method for the special case $\alpha=\beta=k=0$ by
using special functions. In addition, a composite method \cite{Dominguez2} can also be applied to the computation of
this integral for this case, by absorbing the non-oscillatory part of Hankel function into $f$, then interpolating
its product with $f$.  However, as the author in \cite{Xu2} pointed out that the accuracy of this method may becomes worse as the number of Clenshaw--Curtis points
increases and
the fastest convergence of this method obtained is $O(\omega^{-2})$ for fixed number
of Clenshaw--Curtis points.

In view of the advantages of
Clenshaw--Curtis--Filon method, in this paper we will consider a higher order Clenshaw--Curtis--Filon--type method
for the integral \eqref{eq:a1}, which does not require that $f$ is analytic in a enough large region.
As we know, the fast implementation of Clenshaw--Curtis--Filon method largely depends on the
accurate and efficient computation of modified moments.
In addition, the key problem of the efficient computation of the modified moments is
the how to obtain the recurrence relation for them.
Fortunately, we can give a universal method  for the derivation of the recurrence relation for the modified moments.
Moreover, this method can be applied to the modified moments with other type kernels.

The outline of this paper is organized as follows. In Section \ref{sec:2}, we describe the Clenshaw--Curtis--Filon--type
for the integral \eqref{eq:a1}, and present a universal method  for the derivation of the recurrence relation for the modified moments,
by which the modified moments can be efficiently computed with several initial values.
In Section \ref{sec:3}, we give an error bound on $k$ and $\omega$ for the presented method.
Some examples are given in Section \ref{sec:4} to show the efficiency and accuracy.
Finally, we finish this paper in Section \ref{sec:5} by presenting some concluding remarks.

\section{\textbf{Clenshaw--Curtis--Filon--type method and its  implementation}}\label{sec:2}
\setcounter{theorem}{0} \setcounter{equation}{0}
\setcounter{lemma}{0} \setcounter{proposition}{0}
\setcounter{corollary}{0}

In what follows we will consider a Clenshaw--Curtis--Filon--type method for the integrals \eqref{eq:a1}
and its fast implementation. Suppose that $f$ is a sufficiently  function on $[0,1]$, and let
$P_{N+2s}(x)$ denote the Hermite  interpolation polynomial at the Clenshaw--Curtis points
$$x_j=\big(1+\cos(j\pi/N)\big)/2, \ j=0,\ldots,N,$$
where $s$ is a nonnegative integer, and for $\ell=0,\ldots,s$, there holds
\begin{eqnarray}\label{eq:b1}
P_{N+2s}^{(\ell)}(0)=f^{(\ell)}(0),\textrm{\quad}P_{N+2s}(x_j)=f(x_j),\textrm{\quad}P_{N+2s}^{(\ell)}(1)=f^{(\ell)}(1),\textrm{\quad}
j=1,\ldots,N-1.
\end{eqnarray}
Then $P_{N+2s}(x)$ can be written in the following form
\begin{eqnarray}\label{eq:b2}
P_{N+2s}(x)=\sum_{n=0}^{N+2s}a_n T^\ast_n(x),
\end{eqnarray}
where $a_n$ can be fast calculated by fast Fourier transform \cite{Xiang} with $O(N\log N)$ operations, $T^\ast_n(x)$ is the shifted Chebyshev polynomial of the first kind
of degree $n$
on $[0,1]$.

In view of \eqref{eq:b1} and \eqref{eq:b2}, we can define \textbf{Clenshaw--Curtis--Filon--type method }
for the integral \eqref{eq:a1} by
\begin{equation}\label{eq:b3}
Q_{N,s}^{CCF}[f]=\int_0^1P_{N+2s}(x)x^\alpha(1-x)^\beta \E^{\I2kx}H_\nu^{(1)}(\omega x)\D x=\sum_{n=0}^{N+2s}a_nM(n,k,\omega),
\end{equation}
where the modified moments
\begin{equation}\label{eq:b4}
M(n,k,\omega)=\int_0^1x^\alpha(1-x)^\beta T^\ast_n(x) \E^{\I2kx}H_\nu^{(1)}(\omega x)\D x
\end{equation}
have to be computed accurately.
\subsection{\textbf{Recurrence relation for the modified moments}}\label{sec:2.1}
As we have stated in Section \ref{sec:intro}, the key problem of the fast computations of the modified moment $M(n,k,\omega)$
is to obtain a recurrence relation for them. In the following, we will give a universal method for the derivation
of the recurrence relation for the modified moments.

\begin{theorem}\label{th:1}
The modified moments $M(n,k,\omega)$ for $n\geq 4, k\geq 0, \omega>0$ satisfy the following recurrence relation:
\begin{eqnarray}\label{eq:b5}
\nonumber &&\Big(\frac{1}{16}\omega^2-\frac{1}{4}k^2\Big)M(n+4,k,\omega)+f_1(n,\alpha,\beta)M(n+3,k,\omega)+f_2(n,\alpha,\beta)M(n+2,k,\omega)\\
\nonumber &&+f_3(n,\alpha,\beta)M(n+1,k,\omega)+f_4(n,\alpha,\beta)M(n,k,\omega)+f_3(-n,\alpha,\beta)M(n-1,k,\omega)\\
\nonumber &&+f_2(-n,\alpha,\beta)M(n-2,k,\omega)+f_1(-n,\alpha,\beta)M(n-3,k,\omega)\\
&&+\Big(\frac{1}{16}\omega^2-\frac{1}{4}k^2\Big)M(n-4,k,\omega)=0,
\end{eqnarray}
where
\begin{align}
 \label{eq:b6} f_1(n,\alpha,\beta)=&\I k(\alpha+\beta+n+4)-\frac{1}{2}\I k, \\
\nonumber f_2(n,\alpha,\beta)=&9+6(\alpha+\beta+n)+k^2+n^2+\alpha^2+\beta^2-\frac{1}{4}\omega^2-\nu^2\\
\label{eq:b7}&+2(\alpha\beta+\alpha n+\beta n)+\I k(1-2\alpha+2\beta),\\
\nonumber f_3(n,\alpha,\beta)=&2n-8\alpha+12\beta+4(1-\I \alpha k-\I \beta k+\nu^2+\beta n-\alpha n)\\
\label{eq:b8}&-\frac{31}{2}\I k+3\I k(\alpha+\beta-n+4)+4(\beta^2-\alpha^2),\\
\nonumber f_4(n,\alpha,\beta)=&6+4\alpha+12\beta-4\alpha\beta-2\I k+4\I k(\alpha-\beta)\\
\label{eq:b9}&+\frac{3}{8}\omega^2-\frac{3}{2}k^2+6(\alpha^2+\beta^2-\nu^2)-2n^2.
\end{align}
\end{theorem}

\begin{proof}
First, we can rewrite the modified moments $M(n,k,\omega)$ by
\begin{eqnarray}\label{eq:b10}
M(n,k,\omega)=\frac{1}{2^{\alpha+\beta+1}} \E^{\I k}\int_{-1}^1 (1+x)^\alpha (1-x)^\beta T_n(x)\E^{\I kx}H_\nu^{(1)}\big(\frac{1+x}{2} \omega\big)\D x,
\end{eqnarray}
where $T_n(x)$ is the Chebyshev polynomial of degree $n$ of the first kind.

Form the above equality, we can see that the modified  moments $M(n,k,\omega)$ and the integral $\int_{-1}^1 (1+x)^\alpha (1-x)^\beta T_n(x)\E^{\I kx}H_\nu^{(1)}\big(\frac{1+x}{2} \omega\big)\D x$ have the same recurrence relation.

Since the function $y=H_\nu^{(1)}(x)$ satisfies the following Bessel's differential equation \cite[p. 358] {Abram}
\begin{equation}\label{eq:b11}
  x^2\frac{d^2y}{\D x^2}+x\frac{dy}{\D x}+(x^2-\nu^2)=0,
\end{equation}
we have
\begin{eqnarray}\label{eq:b12}
\nonumber&&(1+x)^2\Big[H_\nu^{(1)}\big(\frac{1+x}{2}\omega\big)\Big]^{\prime\prime}+(1+x)\Big[H_\nu^{(1)}\big(\frac{1+x}{2}\omega\big)\Big]^{\prime}\\
&&-\Big(\nu^2-\frac{(1+x)^2\omega^2}{4}\Big)H_\nu^{(1)}\big(\frac{1+x}{2}\omega\big)=0.
\end{eqnarray}
Let
\begin{eqnarray}\label{eq:b13}
K_1=4\int_{-1}^1(1+x)^\alpha(1-x)^\beta(1-x)^2(1+x)^2\E^{\I kx}\Big[H_\nu^{(1)}\big(\frac{1+x}{2}\omega\big)\Big]^{\prime\prime}T_n(x)\D x,
\end{eqnarray}
\begin{eqnarray}\label{eq:b14}
K_2=4\int_{-1}^1(1+x)^\alpha(1-x)^\beta(1-x)^2(1+x)\E^{\I kx}\Big[H_\nu^{(1)}\big(\frac{1+x}{2}\omega\big)\Big]^{\prime}T_n(x)\D x,
\end{eqnarray}
and
\begin{eqnarray}\label{eq:b15}
 K_3=4\int_{-1}^1(1+x)^\alpha(1-x)^\beta(1-x)^2\Big(\nu^2-\frac{(1+x)^2\omega^2}{4}\Big)
\E^{\I kx}H_\nu^{(1)}\big(\frac{1+x}{2}\omega\big)T_n(x)\D x.
\end{eqnarray}
It follows from \eqref{eq:b12} that
\begin{equation}\label{eq:b16}
  K_1+K_2-K3=0.
\end{equation}

Noting that the integrands in $K_1$ and $K_2$ have the common factor $(1-x)^2$ and using integration by
parts, we can easily get
\begin{eqnarray}\label{eq:b17}
K_1=4\int_{-1}^1\big[(1+x)^\alpha(1-x)^\beta(1-x)^2(1+x)^2\E^{\I kx}T_n(x)\big]^{\prime\prime}H_\nu^{(1)}\big(\frac{1+x}{2}\omega\big)\D x,
\end{eqnarray}
\begin{eqnarray}\label{eq:b18}
K_2=4\int_{-1}^1\big[(1+x)^\alpha(1-x)^\beta(1-x)^2(1+x)\E^{\I kx}T_n(x)\big]^{\prime}H_\nu^{(1)}\big(\frac{1+x}{2}\omega\big)\D x.
\end{eqnarray}

According to the properties of the Chebyshev polynomial of the first kind\cite{Mason}
\begin{eqnarray}\nonumber
x^mT_n(x)=2^{-m}\sum_{j=0}^m\binom{m}{j}T_{n+m-2j}(x),
\end{eqnarray}
and
\begin{equation*}
\frac{d}{\D x}T_n(x)=\frac{n}{2}\frac{T_{n-1}(x)-T_{n+1}(x)}{1-x^2},
\end{equation*}
by rewriting the integrands in $K_1, K_2$ and $K_3$ as the sum of the product of Chebyshev
polynomials of different degree and
$(1+x)^\alpha(1-x)^\beta \E^{\I kx}H_\nu^{(1)}\big(\frac{1+x}{2}\omega\big)$, we derive
\begin{eqnarray}\label{eq:b19}
\nonumber K_1&=&-\frac{1}{4}k^2 M(n+4,k,\omega)+f_5(n,\alpha,\beta)M(n+3,k,\omega)+f_6(n,\alpha,\beta)M(n+2,k,\omega)\\
\nonumber &&+f_7(n,\alpha,\beta)M(n+1,k,\omega)+f_8(n,\alpha,\beta)M(n,k,\omega)+f_7(-n,\alpha,\beta)M(n-1,k,\omega)\\
\nonumber &&+f_6(-n,\alpha,\beta)M(n-2,k,\omega)+f_5(-n,\alpha,\beta)M(n-3,k,\omega)\\
&&-\frac{1}{4}k^2M(n-4,k,\omega),
\end{eqnarray}
where
\begin{eqnarray}\label{eq:b20}
f_5(n,\alpha,\beta)&=&\I k(\alpha+\beta+n+4),
\end{eqnarray}
\begin{eqnarray}\label{eq:b21}
\nonumber f_6(n,\alpha,\beta)&=&12+7(\alpha+\beta+n)+k^2+n^2+\alpha^2+\beta^2\\
&&+2(\alpha\beta+\alpha n+\beta n)+2\I k(\beta-\alpha),
\end{eqnarray}
\begin{eqnarray}\label{eq:b22}
\nonumber f_7(n,\alpha,\beta)&=&12(\beta-\alpha)-4\I k(\alpha+\beta)+4(\beta n-\alpha n)\\
&&+3\I k(\alpha+\beta-n+4)+4(\beta^2-\alpha^2),
\end{eqnarray}
\begin{eqnarray}\label{eq:b23}
\nonumber f_8(n,\alpha,\beta)&=&8+10(\alpha+\beta)+4\I k(\alpha-\beta)-2n^2-\frac{3}{2}k^2\\
&&+6(\alpha^2+\beta^2)-4\alpha\beta,
\end{eqnarray}

\begin{eqnarray}\label{eq:b24}
\nonumber K_2&=&-\bigg\{\frac{1}{2}\I kM(n+3,k,\omega)+(\alpha+\beta+n+3-\I k)M(n+2,k,\omega)\\
\nonumber &&-(\frac{1}{2}\I k+4+4\alpha+2n)M(n+1,k,\omega)+(6\alpha-2\beta+2\I k+2)M(n,k,\omega)\\
\nonumber &&-(\frac{1}{2}\I k+4+4\alpha+2n)M(n-1,k,\omega)+(\alpha+\beta-n+3-\I k)M(n-2,k,\omega)\\
 &&+\frac{1}{2}\I kM(n-3,k,\omega)\bigg\},
\end{eqnarray}
\begin{eqnarray}\label{eq:b25}
\nonumber K_3&=&-\frac{1}{16}\bigg\{\omega^2M(n+4,k,\omega)-(4\omega^2+16\nu^2)M(n+2,k,\omega)\\
\nonumber &&+64\nu^2M(n+1,k,\omega)+(6\omega^2-96\nu^2)M(n,k,\omega)+64\nu^2M(n-1,k,\omega)\\
 &&-(4\omega^2+16\nu^2)M(n-2,k,\omega)+\omega^2M(n-4,k,\omega)\bigg\}.
\end{eqnarray}

A combination of the Eqs. \eqref{eq:b16}, \eqref{eq:b19}, \eqref{eq:b24}, \eqref{eq:b25}
leads to recurrence relation \eqref{eq:b5}.
\end{proof}

In the following, let us denote by
\begin{eqnarray}
\label{eq:b26} \widetilde{M}_n^{[1]} &=& \int_0^1 \ln (x)x^\alpha(1-x)^\beta T^\ast_n(x) \E^{\I2kx}H_\nu^{(1)}(\omega x)\D x, \\
\label{eq:b27} \widetilde{M}_n^{[2]} &=& \int_0^1x^\alpha(1-x)^\beta T^\ast_n(x)\ln (1-x) \E^{\I2kx}H_\nu^{(1)}(\omega x)\D x, \\
\label{eq:b28} \widetilde{M}_n^{[3]} &=& \int_0^1\ln (x)x^\alpha(1-x)^\beta \ln (1-x) T^\ast_n(x) \E^{\I2kx}H_\nu^{(1)}(\omega x)\D x,
\end{eqnarray}
respectively, where $\alpha, \beta>-1$. Using the fact that
\begin{eqnarray}
\label{eq:b28b} \widetilde{M}_n^{[1]} &=& \frac{\partial}{\partial \alpha} M(n,k,\omega),\\
\label{eq:b29} \widetilde{M}_n^{[2]} &=& \frac{\partial}{\partial \beta} M(n,k,\omega), \\
\label{eq:b30} \widetilde{M}_n^{[3]} &=& \frac{\partial^2}{\partial\alpha\partial\beta} M(n,k,\omega),
\end{eqnarray}
and according to Theorem \ref{th:1}, we can readily obtain the following result.
\begin{corollary}
The sequences $\widetilde{M}_n^{[\ell]}, \ell=1,2,3$ and $n\geq4, k\geq0, \omega>0$ satisfy the following
ninth-order homogeneous recurrence relations
\begin{eqnarray}\label{eq:31}
 \nonumber &&\big(\frac{1}{16}\omega^2-\frac{1}{4}k^2\big)\widetilde{M}_{n+4}^{[\ell]}+f_1(n,\alpha,\beta)\widetilde{M}_{n+3}+f_2(n,\alpha,\beta)\widetilde{M}_{n+2}^{[\ell]}
+f_3(n,\alpha,\beta)\widetilde{M}_{n+1}^{[\ell]}+\\
\nonumber &&f_4(n,\alpha,\beta)\widetilde{M}_{n}^{[\ell]}+f_3(-n,\alpha,\beta)\widetilde{M}_{n-1}^{[\ell]}+f_2(-n,\alpha,\beta)\widetilde{M}_{n-2}^{[\ell]}+f_1(-n,\alpha,\beta)\widetilde{M}_{n-3}^{[\ell]}\\
&&+\big(\frac{1}{16}\omega^2-\frac{1}{4}k^2\big)\widetilde{M}_{n-4}^{[\ell]}=r_n^{[\ell]},
\end{eqnarray}
where
\begin{eqnarray}\label{eq:b32}
\nonumber r_n^{[1]}&=&-\bigg\{\I kM(n+3,k,\omega)+(6+2\beta+2n+2\alpha+2\I k)M(n+2,k,\omega)-(8+\I k+4n+8\alpha)\\
\nonumber &&M(n+1,k,\omega)+(4-4\beta+4\I k+12\alpha)M(n,k,\omega)-(8+\I k-4n+8\alpha)M(n-1,k,\omega)\\
 &&+(6+2\beta-2n+2\alpha+2\I k)M(n-2,k,\omega)+\I kM(n-3,k,\omega)\bigg\},
\end{eqnarray}
\begin{eqnarray}\label{eq:b33}
\nonumber r_n^{[2]}&=&-\bigg\{\I k M(n+3,k,\omega)+(6+2\beta+2n+2\alpha-2\I k)M(n+2,k,\omega)+(12-\I k+4n+8\beta)\\
\nonumber &&M(n+1,k,\omega)+(12-4\alpha-4\I k+12\beta)M(n,k,\omega)+(12-\I k-4n+8\beta)M(n-1,k,\omega)\\
 &&+(6+2\beta-2n+2\alpha-2\I k)M(n-2,k,\omega)+\I kM(n-3,k,\omega)\bigg\},
\end{eqnarray}
and
\begin{eqnarray}\label{eq:b34}
\nonumber r_n^{[3]}&=&-\bigg\{\I k\big(\widetilde{M}_{n+3}^{[1]}
+\widetilde{M}_{n+3}^{[2]}\big)+(6+2\beta+2n+2\alpha-2\I k)\widetilde{M}_{n+2}^{[1]}+(6+2\beta+2n+2\alpha+2\I k)\widetilde{M}_{n+2}^{[2]}\\
\nonumber &&2M(n+2,k,\omega)+(12-\I k+4n+8\beta)\widetilde{M}_{n+1}^{[1]}+(8+\I k+4n+8\alpha)\widetilde{M}_{n+1}^{[2]}-4M(n,k,\omega)\\
\nonumber  &&+(12-\I k-4n+8\beta)\widetilde{M}_{n-1}^{[1]}+(8+\I k-4n+8\alpha)\widetilde{M}_{n-1}^{[2]}+(6+2\beta-2n+2\alpha-2\I k)\widetilde{M}_{n-2}^{[1]}\\
 &&+(6+2\beta-2n+2\alpha+2\I k)\widetilde{M}_{n-2}^{[2]}+2M(n-2,k,\omega)+\I k\big(\widetilde{M}_{n-3}^{[1]}
+\widetilde{M}_{n-3}^{[2]}\big)\bigg\}.
\end{eqnarray}
\end{corollary}
\begin{remark} \label{re:1}
{\rm
The proof of Theorem \ref{th:1} provides a universal method for the derivations of the recurrence relations
of the modified moments, which can be applied to the modified moments with other kernels that satisfy some
linear differential equations. For example, for the derivations of the recurrence relations of the following three
kinds of modified moments
\begin{eqnarray*}
&&\int_0^1x^\alpha(1-x)^\beta T^\ast_n(x) \E^{\I2kx}{\rm Ai}(-\omega x)\D x,\\
&&\int_0^1x^\alpha(1-x)^\beta T^\ast_n(x) \E^{\I2kx}j_{\nu}(\omega x)\D x,\\
&&\int_0^1x^\alpha(1-x)^\beta T^\ast_n(x) \E^{\I2kx}y_{\nu}(\omega x)\D x,
\end{eqnarray*}
the method is applicable, where ${\rm Ai}(x)$ is Airy function, $j_{\nu}(x),y_{\nu}(x)$ are  spherical Bessel functions
of the first kind and second kind\cite{Abram}, respectively.
Moreover,  by differentiating the  recurrence relation  with respect to parameters $\alpha, \beta$, one can also
obtain the recurrence relations for the modified moments with logarithmic singularities at two
endpoints. As this idea is tangential
to the topic of this paper, we will not study it further.
}
\end{remark}

\begin{remark}\label{re:2}
{\rm
For $\omega=2k$, the coefficients of $M(n+4,k,\omega)$ and $M(n-4,k,\omega)$ are both zero, then the recurrence relation \eqref{eq:b5} reduces to
a seven-term recurrence relation.
}
\end{remark}

\subsection{\textbf{Fast computations of the modified moments}}\label{sec:2.2}

In what follows we will be concerned with the fast computation of the modified moments by using the recurrence relation
\eqref{eq:b5}. According to the symmetry of the recurrence relation of the Chebyshev polynomials $T_n^\ast (x)$, it is
convenient to define $T_{-n}^\ast (x)=T_n^\ast (x)$ for $n=1,2,3,\ldots$. Consequently,
\begin{eqnarray*}
M(-n,k,\omega)=M(n,k,\omega), \quad k=1,2,3,\ldots.
\end{eqnarray*}
Moreover, It can be shown that \eqref{eq:b5} is valid,  not only for $n\geq 4$, but also for all integers of $n$.

Unfortunately,  the application of recurrence relations in the forward direction is not
always numerically stable. Practical experiments show that the modified moments $M(n,k,\omega), n=0,1,2,\ldots$
can be computed accurately by using the recurrence relation
\eqref{eq:b5} as long as $n\leq (k+\omega/2)$. However, for $n> (k+\omega/2)$,
forward recursion is no longer applicable due to the loss of significant figures increases. In
this case, \eqref{eq:b5} has to be solved as a boundary value problem.
Fortunately, we can use Oliver's algorithm \cite{Oliver} or Lozier's algorithm \cite{Lozier} to solve this problem for the modified moments
with five starting moments and three
end moments.
Particularly, for Lozier's algorithm, we can set three end moments to zero.
Also, this algorithm incorporates a numerical
test for determining the optimum location of the endpoint. The advantage is
that a user-required accuracy is automatically obtained, without computation of the asymptotic expansion.
In conclusion,  several starting values for the modified moments for forward recursion and Oliver's algorithm or Lozier's algorithm
are needed. In addition, the three
end moments can be computed by using asymptotic expansion in \cite{Erdelyi} or the method in \cite{He}.

Since the shifted Chebyshev polynomials $T_n^\ast(x)$ can be rewritten in terms of powers of $x$, the
five starting modified moments can be computed by the following formulas
\begin{align*}
M(0,k,\omega)=&I(0,k,\omega)\\
M(1,k,\omega)=&2I(1,k,\omega)-I(0,k,\omega)\\
M(2,k,\omega)=&8I(2,k,\omega)-8I(1,k,\omega)+I(0,k,\omega)\\
M(3,k,\omega)=&32I(3,k,\omega)-48I(2,k,\omega)+18I(1,k,\omega)-I(0,k,\omega)\\
M(4,k,\omega)=&128I(4,k,\omega)-256I(3,k,\omega)+160I(2,k,\omega)-32I(1,k,\omega)+I(0,k,\omega),
\end{align*}
where
\begin{eqnarray}\label{eq:b35}
I(j,k,\omega)=\int_0^1x^{\alpha+j}(1-x)^\beta \E^{\I2k x}H_\nu^{(1)}(\omega x)\D x,
\end{eqnarray}
which can be efficiently computed by the method in \cite{He} with small number of
points.

For a special case $\omega=2k$, the computation of the integral \eqref{eq:b35} is reduced
to the evaluation of
\begin{eqnarray*}
\widehat{I}(\alpha,\beta,\nu,\omega)=\int_0^1x^\alpha(1-x)^\beta \E^{\I \omega x}H_\nu^{(1)}(\omega x)\D x, \,\, \alpha>-1, \beta>-1,
\end{eqnarray*}
which can also be accurately computed through the following theorem.
\begin{theorem}
For all
$\alpha>-1, \beta>-1$ and $\omega>0$, it holds that
\begin{eqnarray}\label{eq:b36}
\widehat{I}(\alpha,\beta,\nu,\omega)=I_1+\I(I_2+I_3)-I_4,
\end{eqnarray}
where
\begin{eqnarray}\label{eq:b37}
I_1=C G^{1,4}_{4,6}\left(\begin{array}{c}-\frac{\alpha}{2}, \frac{1-\alpha}{2},\frac{1}{4},\frac{3}{4}
\\ \frac{\nu}{2},-\frac{\nu}{2},\frac{1+\nu}{2},
\frac{1-\nu}{2},-\frac{\alpha+\beta+1}{2}, -\frac{\alpha+\beta}{2}
\end{array}\Bigg|\omega^2 \right),
\end{eqnarray}
\begin{eqnarray}\label{eq:b38}
I_2=\omega C G^{1,4}_{4,6}\left(\begin{array}{c}-\frac{\alpha+1}{2}, -\frac{\alpha}{2},-\frac{1}{4},\frac{1}{4}
\\ \frac{\nu}{2},-\frac{\nu}{2},-\frac{1-\nu}{2},
-\frac{1+\nu}{2},-\frac{\alpha+\beta+2}{2}, -\frac{\alpha+\beta+1}{2}
\end{array}\Bigg|\omega^2 \right),
\end{eqnarray}
\begin{eqnarray}\label{eq:b39}
I_3=-C G^{2,4}_{5,7}\left(\begin{array}{c}-\frac{\alpha}{2}, \frac{1-\alpha}{2},\frac{1}{4},\frac{3}{4},\frac{1-\nu}{2}
\\ -\frac{\nu}{2},\frac{\nu}{2},\frac{1+\nu}{2},\frac{1-\nu}{2},
\frac{1-\nu}{2},-\frac{\alpha+\beta+1}{2}, -\frac{\alpha+\beta}{2}
\end{array}\Bigg|\omega^2 \right),
\end{eqnarray}
\begin{eqnarray}\label{eq:b40}
I_4=-\omega CG^{2,4}_{5,7}\left(\begin{array}{c}-\frac{\alpha+1}{2}, -\frac{\alpha}{2},-\frac{1}{4},\frac{1}{4},\frac{1-\nu}{2}
\\ -\frac{\nu}{2},\frac{\nu}{2},\frac{1-\nu}{2},-\frac{1+\nu}{2},
\frac{\nu-1}{2},-\frac{\alpha+\beta+2}{2}, -\frac{\alpha+\beta+1}{2}
\end{array}\Bigg|\omega^2 \right),
\end{eqnarray}
 and
\begin{equation}\nonumber
  \label{eq:b41}
  G^{m,n}_{p,q}\left(\begin{array}{c} a_1,\ldots, a_n, a_{n+1}, \ldots, a_p \\ b_1,\ldots, b_m, b_{m+1}, \ldots, b_q \end{array}\Bigg| z\right) = \frac 1 {2\pi i} \oint_L \frac{\prod_{k=1}^m \Gamma(b_k-s) \prod_{j=1}^n \Gamma(1-a_j+s)}{\prod_{k=m+1}^q \Gamma(1-b_k+s) \prod_{j=n+1}^p \Gamma(a_j-s)} z^s ds
\end{equation}
is Meijer G--function \cite{Bateman}, $C=2^{-(\beta+\frac{3}{2})}\Gamma(\beta+1)$.
\end{theorem}

\begin{proof}
Substituting $H_\nu^{(1)}(x)=J_\nu(x)+\I Y_\nu(x)$ into $\widehat{I}(\alpha,\beta,\nu,\omega)$ gives
\begin{eqnarray}\label{eq:b42}
\nonumber &&I(\alpha,\beta,\nu,\omega)\\
\nonumber &=&\int_0^1x^\alpha(1-x)^\beta \cos(\omega x)J_\nu(\omega x)\D x+\I\int_0^1x^\alpha(1-x)^\beta \sin(\omega x)J_\nu(\omega x)\D x+\\
&&\I\int_0^1x^\alpha(1-x)^\beta \cos(\omega x)Y_\nu(\omega x)\D x-\int_0^1x^\alpha(1-x)^\beta \sin(\omega x)Y_\nu(\omega x)\D x.
\end{eqnarray}
Note that \cite{besselj, bessely}
\begin{eqnarray}\label{eq:b43}
\nonumber &&\leftidx{_0}{F}{_1}(;b;-\frac{z^2}{4})J_\nu(z)\\
&=&\frac{\Gamma(b)}{\sqrt{\pi}}2^{b-1}G^{1,2}_{2,4}\left(\begin{array}{c}\frac{1-b}{2}, 1-\frac{b}{2}
\\ -\frac{\nu}{2},\frac{\nu}{2},1-b+\frac{\nu}{2},1-b-\frac{\nu}{2}
\end{array}\Bigg|z^2 \right),
\end{eqnarray}
\begin{eqnarray}\label{eq:b44}
\nonumber &&\leftidx{_0}{F}{_1}(;b;-\frac{z^2}{4})Y_\nu(z)\\
&=&\frac{\Gamma(b)}{\sqrt{\pi}}2^{b-1}G^{1,2}_{2,4}\left(\begin{array}{c}\frac{1-b}{2}, 1-\frac{b}{2}, \frac{1-\nu}{2}
\\ -\frac{\nu}{2},\frac{\nu}{2},\frac{1-\nu}{2}, 1-b+\frac{\nu}{2},1-b-\frac{\nu}{2}
\end{array}\Bigg|z^2 \right).
\end{eqnarray}
On the other hand, there holds \cite{Meijer}
\begin{eqnarray}\label{eq:b45}
\nonumber & &\int_0^x t^{\alpha-1}(x-t)^{\beta-1}G^{m,n}_{p,q}\left(\begin{array}{c} a_1\ldots a_n, a_{n+1} \ldots a_p \\ b_1\ldots b_m, b_{m+1} \ldots b_q \end{array}\Bigg|\omega t^{l}\right)dt \\
&=&\frac{l^{-\beta}\Gamma(\beta)}{x^{1-\alpha-\beta}}G^{m,n+l}_{p+l,q+l}\left(\begin{array}{c} \frac{1-\alpha}{l},\ldots, \frac{l-\alpha}{l}, a_1\ldots a_n, a_{n+1} \ldots a_p \\ b_1\ldots b_m, b_{m+1} \ldots b_q, \frac{1-\alpha-\beta}{l},\ldots, \frac{l-\alpha-\beta}{l} \end{array}\Bigg|\omega x^{l}\right),
\end{eqnarray}
and \cite{Luke}
\begin{eqnarray}\label{eq:b46}
\cos(z)=\leftidx{_0}{F}{_1}\left(;\frac{1}{2};-\frac{z^2}{4}\right),\quad \sin(z)=\leftidx{_0}{F}{_1}\left(;\frac{3}{2};-\frac{z^2}{4}\right),
\end{eqnarray}

According to \eqref{eq:b42}, and setting $b=1/2, 3/2$ in \eqref{eq:b43}--\eqref{eq:b44}, respectively,  then substituting
them into \eqref{eq:b45}, we can easily derive the result \eqref{eq:b36}.
\end{proof}

\begin{remark}\label{re:3}
{\rm We choose $10$ points for the Gauss-type method in \cite{He} to evaluate $I(j,k,\omega), j=0,1,2,3,4$
for $\omega\neq 2k$. While for $\omega=2k$, we compute them by using  the formula \eqref{eq:b36} through
Meijer G--function, which  can be efficiently computed with the {\sc Matlab} code \textbf{MeijerG.m} \cite{Oreshkin}.
}
\end{remark}

\section{\textbf{Error estimate about $k$ and $\omega$ for the method \eqref{eq:b3}}}\label{sec:3}
\setcounter{theorem}{0} \setcounter{equation}{0}
\setcounter{lemma}{0} \setcounter{proposition}{0}
\setcounter{corollary}{0}
In \cite{Sloan1, Sloan2}, Sloan and Smith presented a product-integration rule with the Clenshaw--Curtis points
for approximating the integral $\int_{-1}^1k(x)f(x)dx$, where $k(x)$ is integrable
and $f(x)$ is continuous. Moreover, the authors also considered the theoretical convergence properties of the method,
and obtained the satisfactory rates of convergence for all continuous functions $f(x)$, if $k(x)$
satisfies  $\int_{-1}^{1}|k(x)|^p dx<\infty$ for some $p>1$. Since
$$\int_{0}^{1}\big|x^\alpha(1-x)^\beta \E^{\I2kx}H_\nu^{(1)}(\omega x)\D \big|^p dx< \infty,$$
for all $p>1$
from \cite{Sloan1, Sloan2}, we see that the Clenshaw--Curtis--Filon--type method (\ref{eq:b3}) for integral (\ref{eq:a1})
is uniformly convergent in $N$ for fixed $k$ and $\omega$, that is
\begin{equation}\nonumber
\lim_{N\rightarrow \infty}\int_0^1P_{N+2s}(x)x^\alpha(1-x)^\beta \E^{\I2kx}H_\nu^{(1)}(\omega x)\D x=\int_0^1f(x)x^\alpha(1-x)^\beta \E^{\I2kx}H_\nu^{(1)}(\omega x)\D x.
\end{equation}
In the what follows we will consider the error estimate on $k$ and $\omega$ for the method \eqref{eq:b3}.
To obtain an error bound for method \eqref{eq:b3}, we first
introduce the following theorem.

\begin{theorem}\label{th:2}
For each $\alpha>-1, \beta>-1$, the asymptotics of the integral $\int_0^1x^{\alpha}(1-x)^\beta \E^{\I2k x}H_\nu^{(1)}(\omega x)\D x$
can be estimated by the following  three formulas.

(i) If $k$ is fixed and $\omega\rightarrow \infty$, there holds
\begin{eqnarray}\label{eq:c1}
\int_0^1x^{\alpha}(1-x)^\beta \E^{\I2k x}H_\nu^{(1)}(\omega x)\D x= O\left(\frac{1}{\omega^{1+\tau_1}}\right).
\end{eqnarray}
where $\tau_1=\min\left\{\alpha,\beta\right\}$.

(ii)  If $\omega$ is fixed and $k \rightarrow \infty$, there holds
\begin{eqnarray}\label{eq:c2}
\int_0^1x^{\alpha}(1-x)^\beta \E^{\I2k x}H_\nu^{(1)}(\omega x)\D x= \left\{
   \begin{array}{ll}
     O\left(\frac{1+\ln\left(k\right)}{k^{1+\alpha}}\right), & \hbox{$\nu=0, \alpha\leq\beta$}, \\
     O\left(\frac{1}{k^{1+\beta}}\right), & \hbox{$\nu=0, \alpha<\beta$}, \\
    O\left(\frac{1}{k^{1+\tau_2}}\right), & \hbox{$\nu\neq0$},
   \end{array}
 \right.
\end{eqnarray}
where $\tau_2=\min\left\{\alpha-|\nu|,\beta\right\}$.

(iii) If $\omega=2k$ and $\omega \rightarrow \infty$, there holds
\begin{eqnarray}\label{eq:c3}
\int_0^1x^{\alpha}(1-x)^\beta \E^{\I \omega x}H_\nu^{(1)}(\omega x)\D x=O\left(\frac{1}{\omega^{1+\tau_1}}\right).
\end{eqnarray}
\end{theorem}

\begin{proof}
By using the complex integration theory and substituting the original interval of integration by the
paths of steepest descent for the integral,
we can rewrite the integral $\int_0^1x^{\alpha}(1-x)^\beta \E^{\I2k x}H_\nu^{(1)}(\omega x)\D x$ as a sum of two line
integrals (which is a special case of Eq. (20) in \cite{He} with $f(x)=1, b=1$), that is
\begin{eqnarray}\label{eq:c4}
\int_0^1x^{\alpha}(1-x)^\beta \E^{\I2k x}H_\nu^{(1)}(\omega x)\D x=L_0[f]-L_1[f]
\end{eqnarray}
where
\begin{eqnarray}\label{eq:c5}
 L_0[f]=\frac{2\I^\alpha}{\I^\nu \pi(2k+\omega)^{1+\alpha}}\int_0^\infty\left(1-\frac{\I x}{2k+\omega}\right)^\beta K_\nu\left(\frac{\omega x}{2k+\omega}\right)
x^\alpha\E^{-2kx/(2k+\omega)}\D x,
\end{eqnarray}
\begin{eqnarray}\label{eq:c6}
 L_1[f]=\frac{(-\I)^\beta \I \E^{\I \omega}}{\pi(2k+\omega)^{1+\beta}}\int_0^\infty\left(1+\frac{\I x}{2k+\omega}\right)^\alpha  H_\nu^{(1)}\left(\omega+\frac{\I \omega x}{2k+\omega}\right)
\E^{\omega x/(2k+\omega)}x^\beta \E^{-x}\D x,
\end{eqnarray}
here, $K_\nu(x)$ is
the modified Bessel function of the second kind of order $\nu$ \cite{Abram}.

According to the Theorem in \cite{Bleistein}, when $\omega\rightarrow\infty$, for every fixed $k$, we have
\begin{eqnarray}\label{eq:c6b}
L_0[f]=O\left(\frac{1}{\omega^{1+\alpha}}\right), \quad L_1[f]=O\left(\frac{1}{\omega^{1+\beta}}\right),
\end{eqnarray}
which leads to \eqref{eq:c1} directly.

On the other hand, when $k\rightarrow\infty$, for every fixed $\omega$, we have
\begin{eqnarray}
\label{eq:c7}L_0[f]&=& \left\{
   \begin{array}{ll}
     O\left(\frac{1+\ln\left(k\right)}{k^{1+\alpha}}\right), & \hbox{$\nu=0, \alpha\leq\beta$}, \\
    O\left(\frac{1}{k^{1-|\nu|+\alpha}}\right), & \hbox{$\nu\neq0$},
   \end{array}
 \right.\\
\label{eq:c8} L_1[f]&=&O\left(\frac{1}{k^{1+\beta}}\right),
\end{eqnarray}
which derives \eqref{eq:c2} directly.

Eq. \eqref{eq:c3} can be derived by a similar way to the proof of \eqref{eq:c1}.
This complete the proof.

\end{proof}

{\bf Example 3.1.}  Let us consider the asymptotics of the integral
\begin{eqnarray}
  \label{eq:c9}\widetilde{I}_1(\alpha,\beta,\omega) &=&  \int_0^1x^{\alpha}(1-x)^\beta \E^{\I 20 x}H_\nu^{(1)}(\omega x)\D x.
\end{eqnarray}

{\bf Example 3.2.} Let us consider the asymptotics of the integral
\begin{eqnarray}
  \label{eq:c10}\widetilde{I}_2(\alpha,\beta,k)&=&  \int_0^1x^{\alpha}(1-x)^\beta \E^{\I 2k x}H_\nu^{(1)}(10 x)\D x.
\end{eqnarray}

{\bf Example 3.3.} Let us consider the asymptotics of the integral
\begin{eqnarray}
   \label{eq:c11}\widetilde{I}_3(\alpha,\beta,\omega)&=&  \int_0^1x^{\alpha}(1-x)^\beta \E^{\I \omega x}H_\nu^{(1)}(\omega x)\D x.
\end{eqnarray}
From Figs. \ref{fig:1}--\ref{fig:3}, we see that the asymptotic orders on $k$ and $\omega$ stated in Theorem \ref{th:2} are attainable.

According to Theorem \ref{th:2}, we can easily obtain the error bound for the Clenshaw--Curtis--Filon--type method \eqref{eq:b3},
by using the technique of Theorem 3.1 in \cite{Xu3}.

\begin{figure}[t]
 \center
\includegraphics[width=0.46\textwidth, height=0.18\textheight]{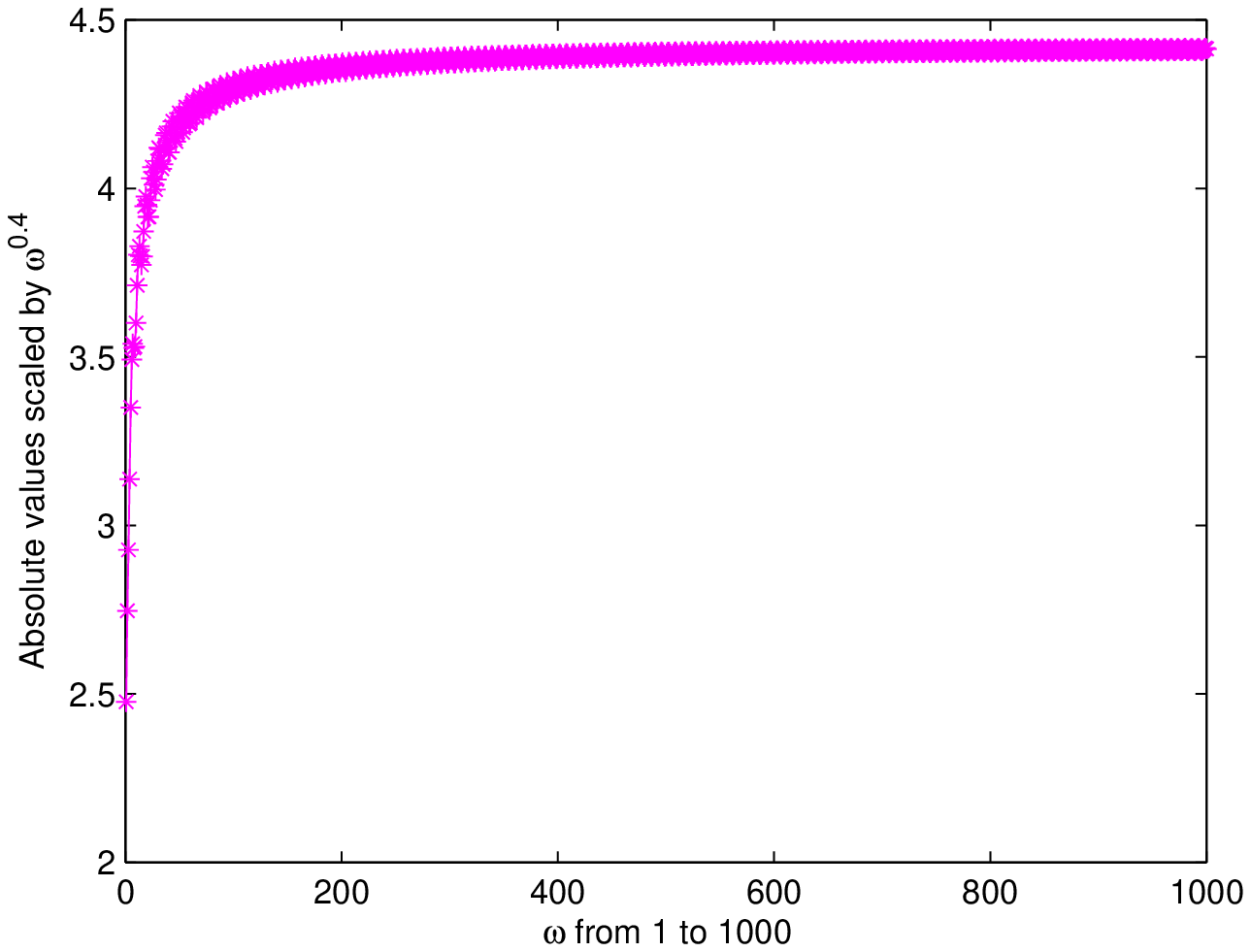}
\includegraphics[width=0.46\textwidth, height=0.18\textheight]{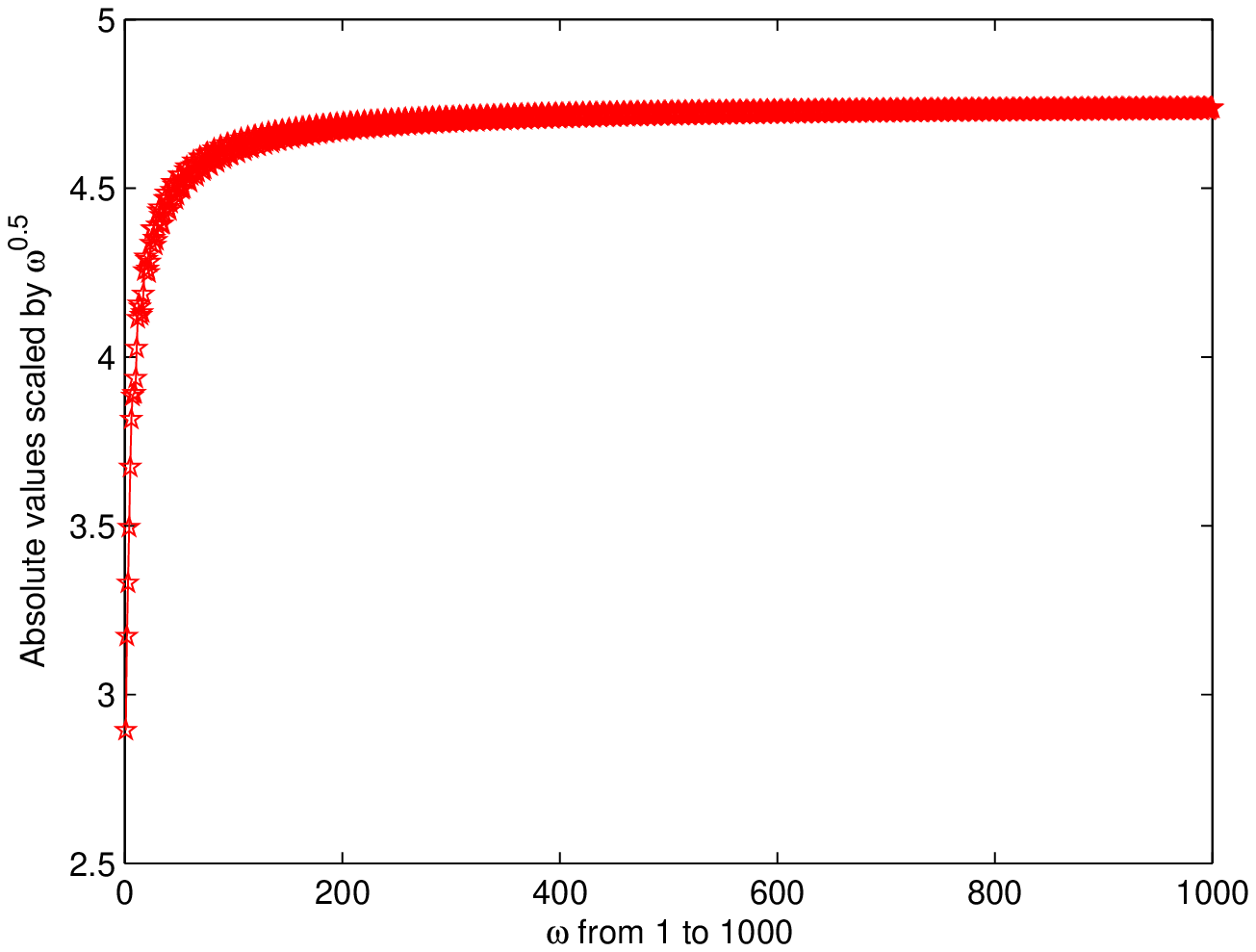}
\caption{Absolute values of \eqref{eq:c9} scaled by $\omega^{0.4}$ with $\nu=0,\alpha=-0.6,\beta=-0.3$ (left), and $\omega^{0.5}$ with $\nu=0.3,\alpha=-0.5,\beta=-0.2$ (right), respectively, when $\omega$ runs from $1$ to $1000$.}\label{fig:1}
\end{figure}

\begin{figure}[t]
 \center
\includegraphics[width=0.46\textwidth, height=0.18\textheight]{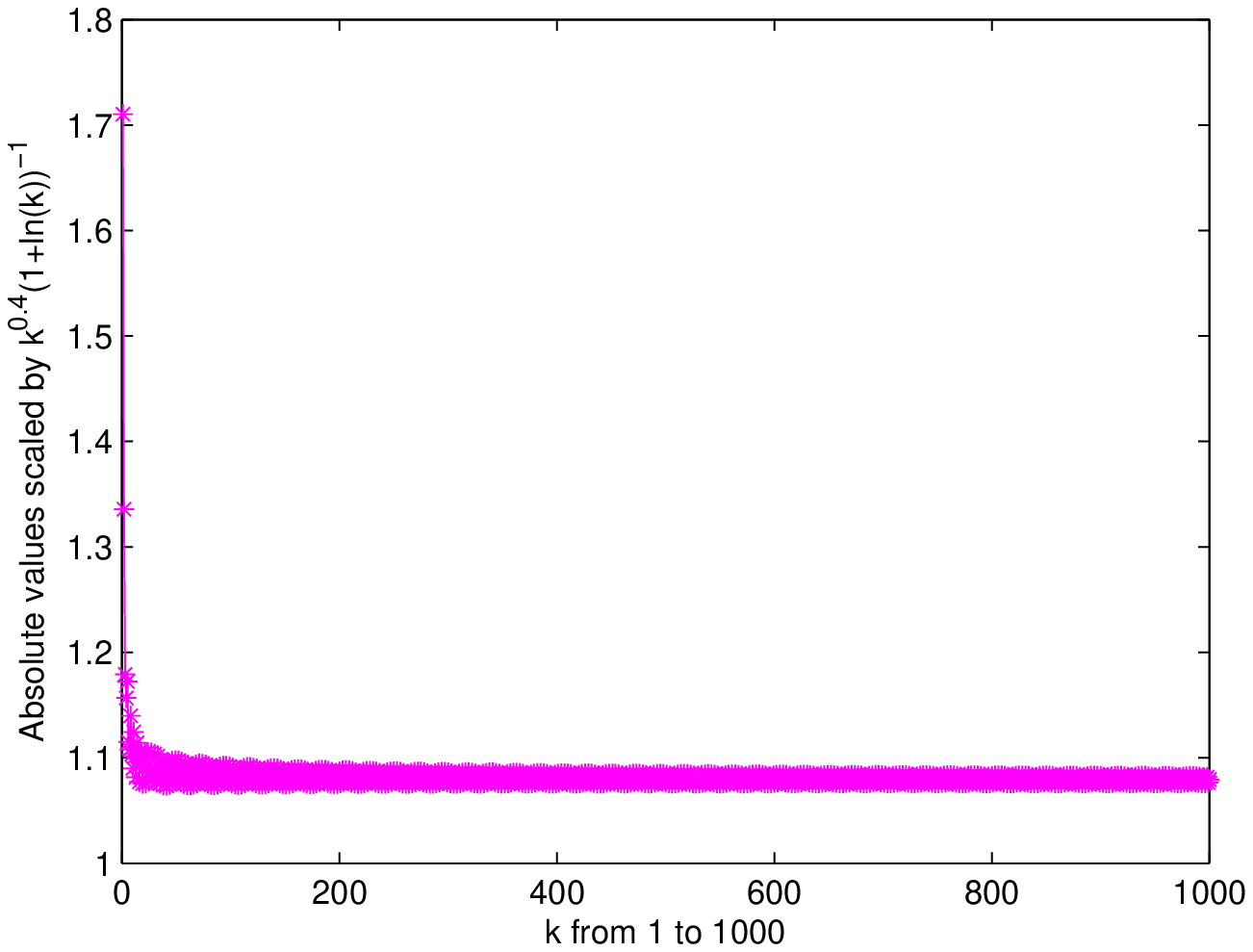}
\includegraphics[width=0.46\textwidth, height=0.18\textheight]{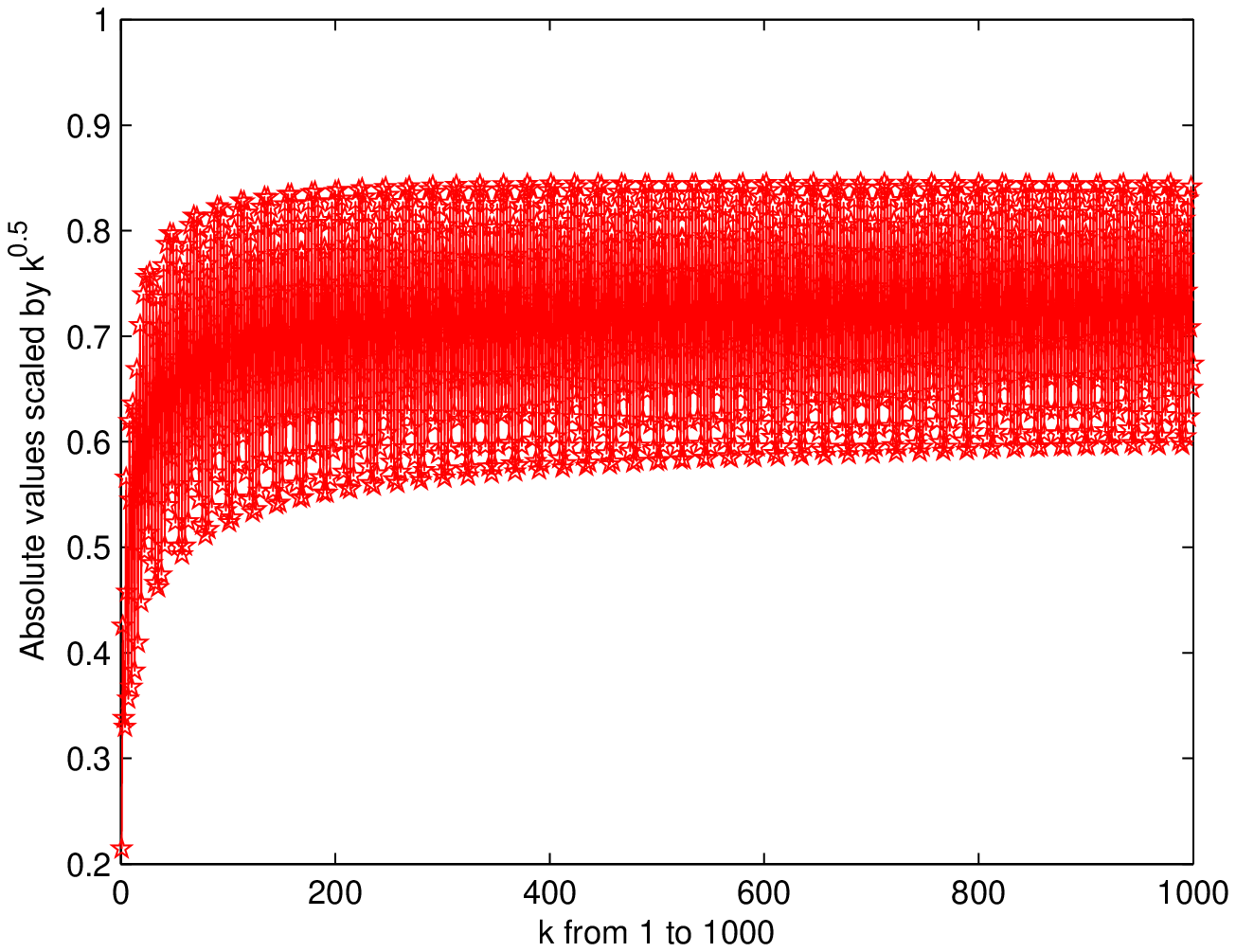}
\caption{Absolute values of \eqref{eq:c10} scaled by $k^{0.4}\left(1+\ln(k)\right)^{-1}$ with $\nu=0,\alpha=-0.6,\beta=-0.3$ (left), and $k^{0.5}$ with $\nu=0.3,\alpha=-0.2,\beta=-0.4$ (right), respectively, when $k$ runs from $1$ to $1000$.}\label{fig:2}
\end{figure}

\begin{figure}[t]
 \center
\includegraphics[width=0.46\textwidth, height=0.18\textheight]{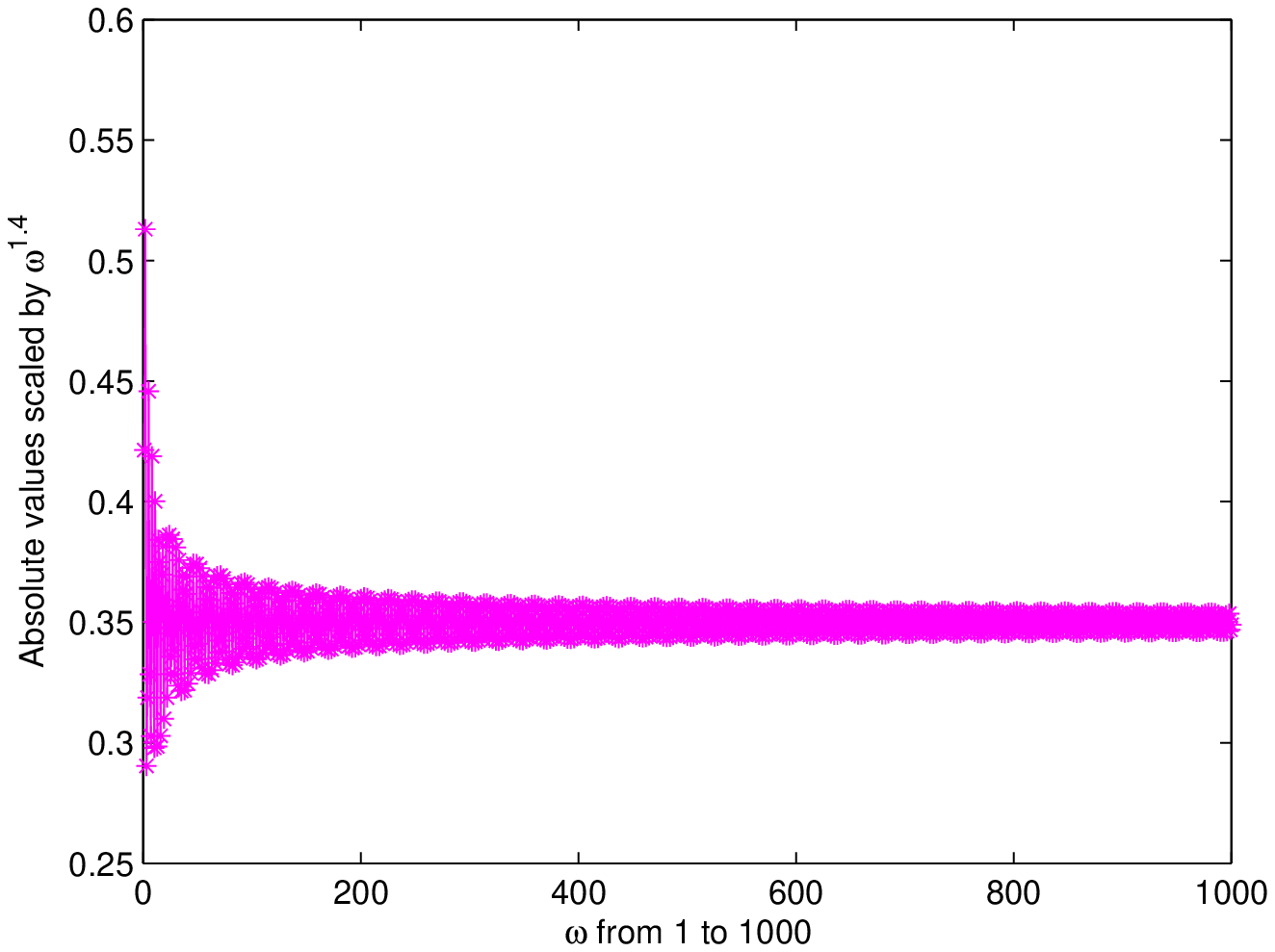}
\includegraphics[width=0.46\textwidth, height=0.18\textheight]{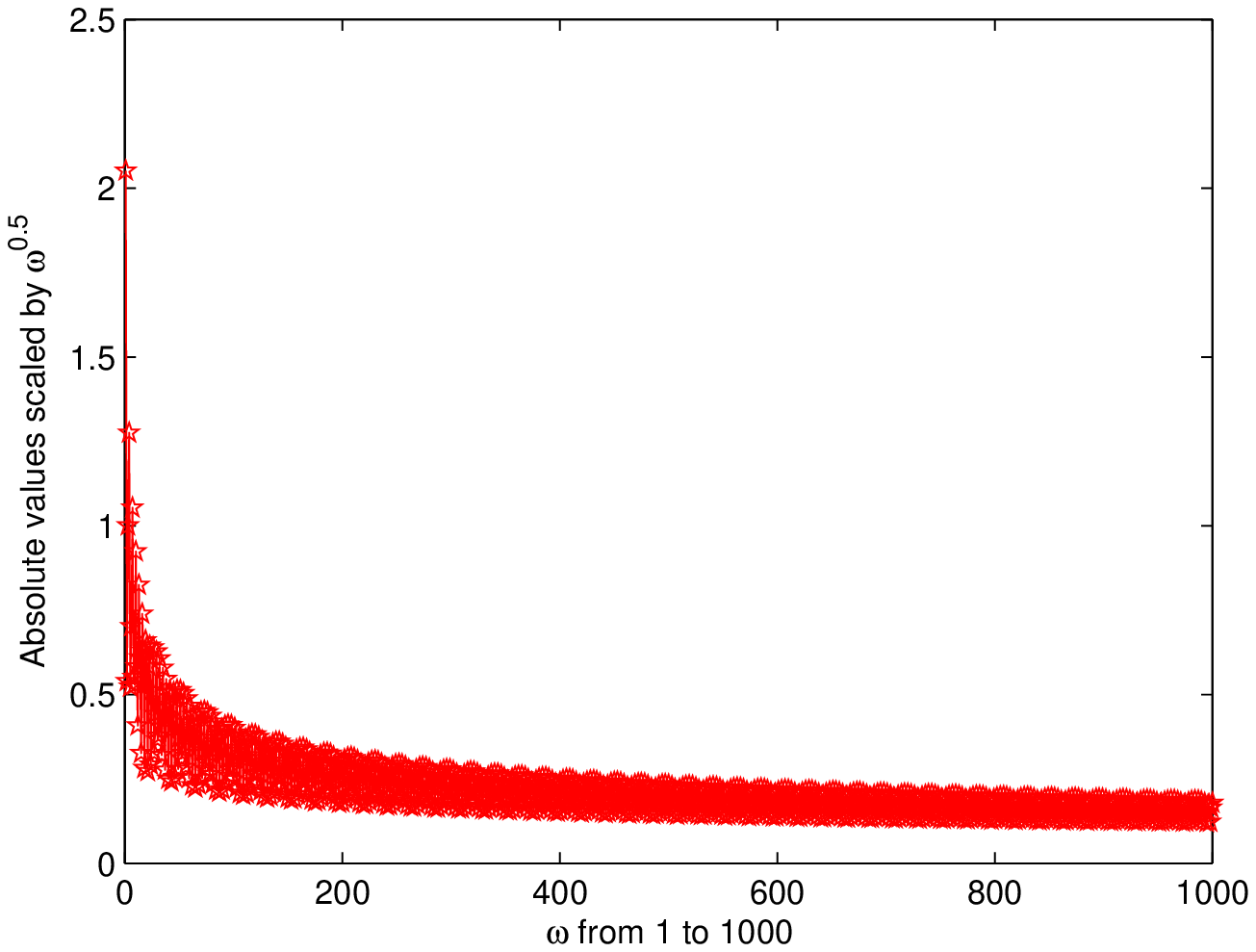}
\caption{Absolute values of \eqref{eq:c11} scaled by $\omega^{1.4}$ with $\nu=0,\alpha=0.4,\beta=0.5$ (left), and $\omega^{0.5}$ with $\nu=0.3,\alpha=-0.2,\beta=-0.5$ (right), respectively, when $\omega$ runs from $1$ to $1000$.}\label{fig:3}
\end{figure}

\begin{theorem}\label{th:3}
Suppose that $f(x)$ is a sufficiently smooth function on $[0,1]$, then for each $\alpha-|\nu|>-1, \beta>-1$ and fixed $N$, the error bound on $k$ and
$\omega$ for the Clenshaw--Curtis--Filon--type method \eqref{eq:b3}
for the integral \eqref{eq:a1} can be estimated by the following  three formulas.

(i) For fixed $k$, when $\omega\rightarrow \infty$, there holds
\begin{eqnarray}\label{eq:c12}
I[f]-Q_{N,s}^{CCF}[f]= O\left(\frac{1}{\omega^{s+2+\tau_1}}\right).
\end{eqnarray}
where $\tau_1=\min\left\{\alpha,\beta\right\}$.

(ii) For fixed $\omega$, when $k \rightarrow \infty$, there holds
\begin{eqnarray}\label{eq:c13}
I[f]-Q_{N,s}^{CCF}[f]= \left\{
     \begin{array}{ll}
     O\left(\frac{1+\ln\left(k\right)}{k^{s+2+\alpha}}\right), & \hbox{$\nu=0, \alpha\leq\beta$}, \\
     O\left(\frac{1}{k^{s+2+\beta}}\right), & \hbox{$\nu=0, \alpha<\beta$}, \\
    O\left(\frac{1}{k^{s+2+\tau_2}}\right), & \hbox{$\nu\neq0$},
   \end{array}
 \right.
\end{eqnarray}
where $\tau_2=\min\left\{\alpha-|\nu|,\beta\right\}$.

(iii) For a special case that $\omega=2k$,  when $\omega \rightarrow \infty$, there holds
\begin{eqnarray}\label{eq:c14}
I[f]-Q_{N,s}^{CCF}[f]= O\left(\frac{1}{\omega^{s+2+\tau_1}}\right).
\end{eqnarray}
\end{theorem}

\section{\textbf{Numerical examples}}\label{sec:4}
\setcounter{theorem}{0} \setcounter{equation}{0}
\setcounter{lemma}{0} \setcounter{proposition}{0}
\setcounter{corollary}{0}

In this section, we will present several examples to illustrate the efficiency and accuracy of the proposed method.
Throughout the paper, all numerical computations were implemented on the R2012a version of the {\sc Matlab} system.
The experiments were performed on a computer with 3.20 GHz processor and 4 GB of RAM.
In addition, the exact values of all the considered integrals $I[f]$ were computed in the
{\sc Maple 17}  using 32 decimal digits precision arithmetic.

\begin{figure}[hbtp]
 \center
\includegraphics[width=0.46\textwidth, height=0.18\textheight]{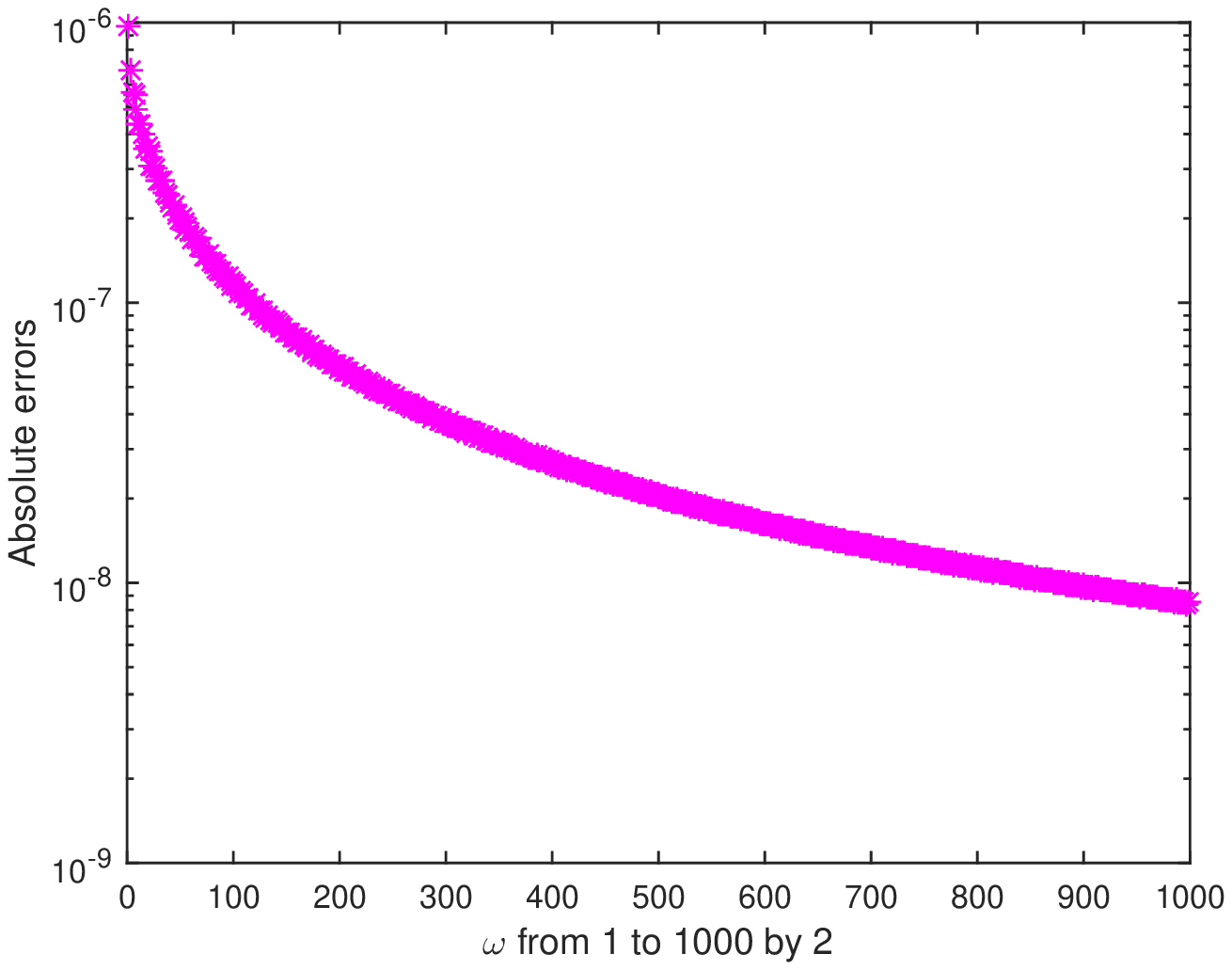}
\includegraphics[width=0.46\textwidth, height=0.18\textheight]{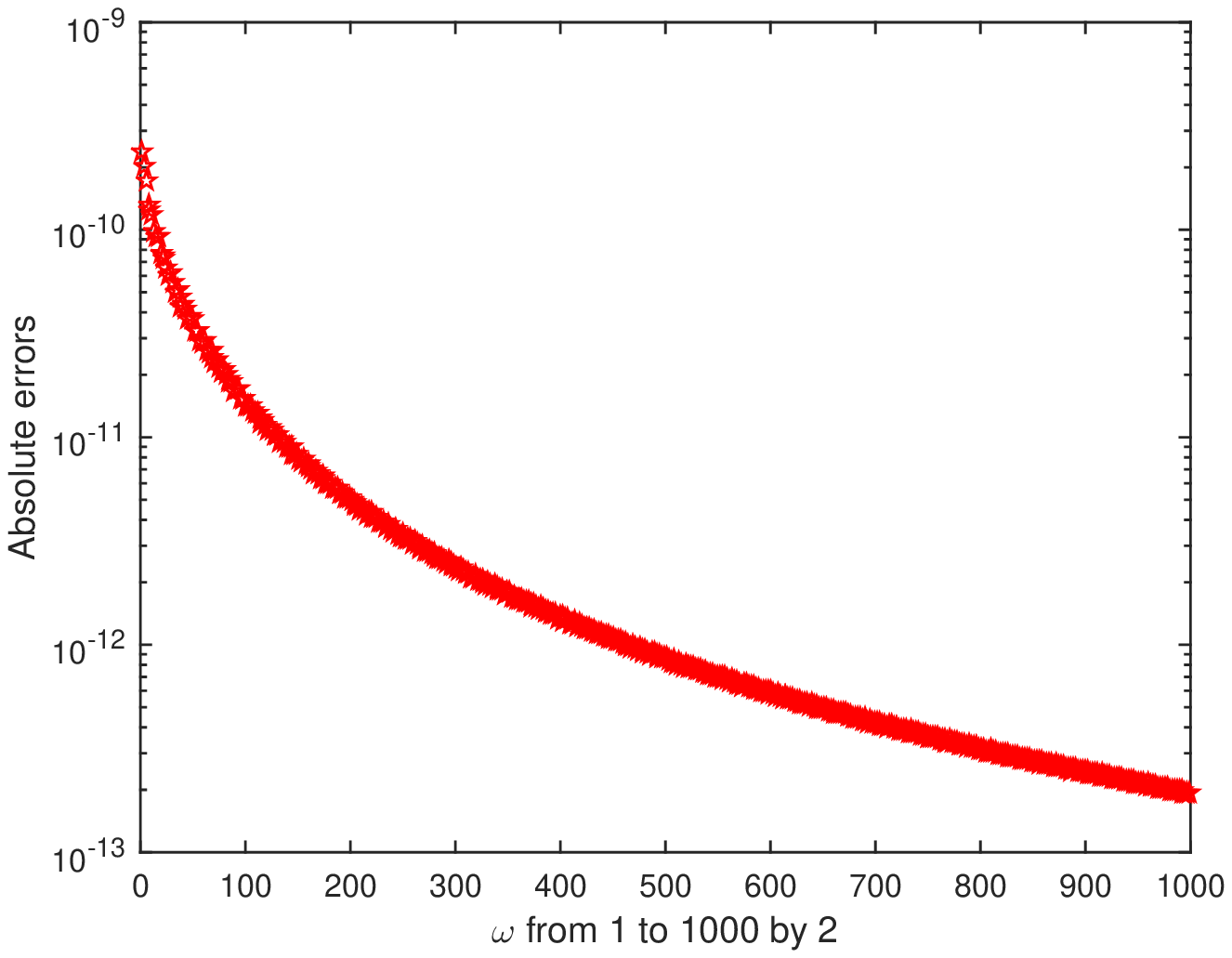}
\caption{Absolute errors for the Clenshaw--Curtis--Filon--type method for the integral \eqref{eq:d1} with $s=0$ (left) and $s=1$ (right), when
$N=4$, $k=50$, $\omega$  from $1$ to $1000$ by 2.}\label{fig:4}
\end{figure}
\begin{figure}[hbtp]
 \center
\includegraphics[width=0.46\textwidth, height=0.18\textheight]{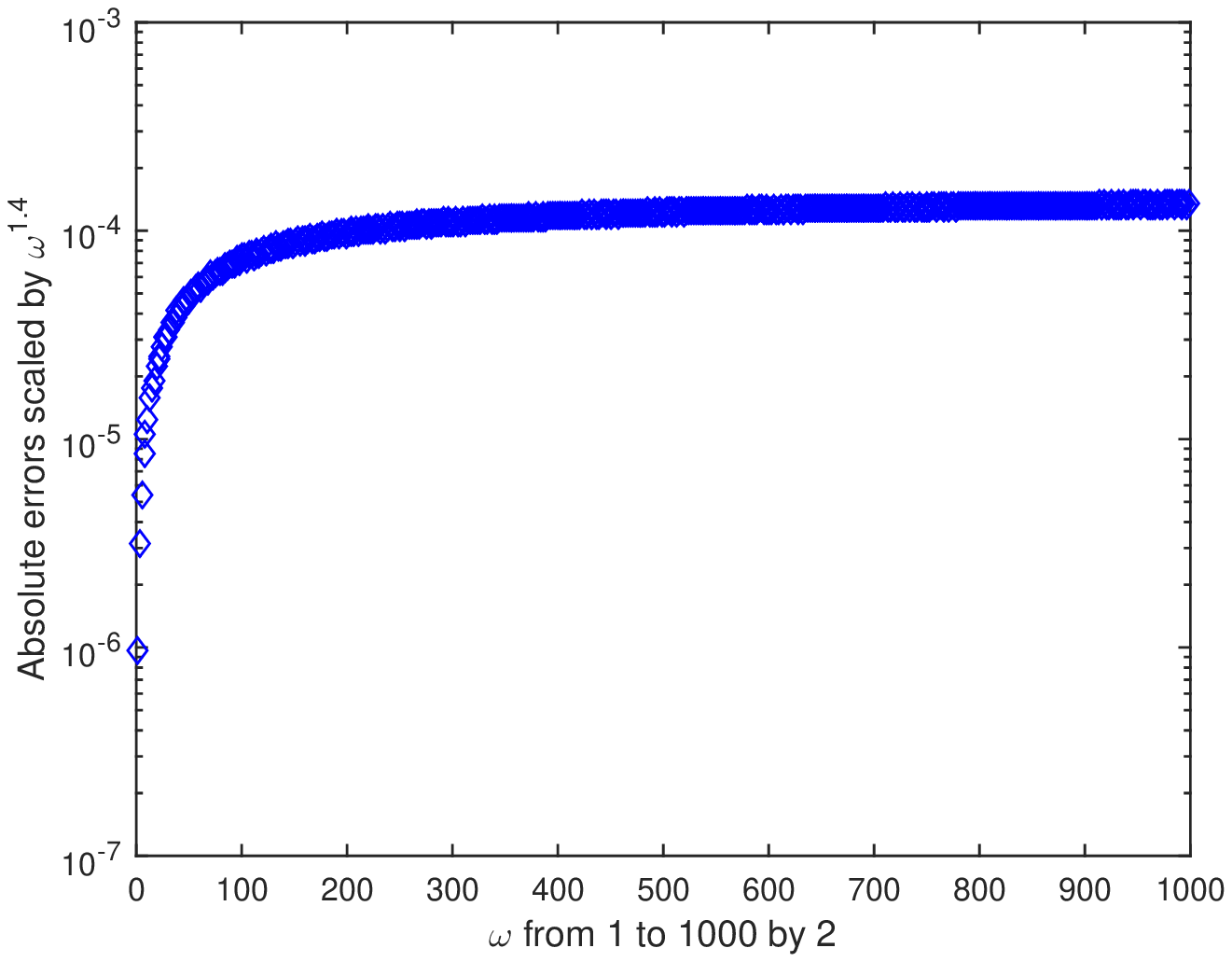}
\includegraphics[width=0.46\textwidth, height=0.18\textheight]{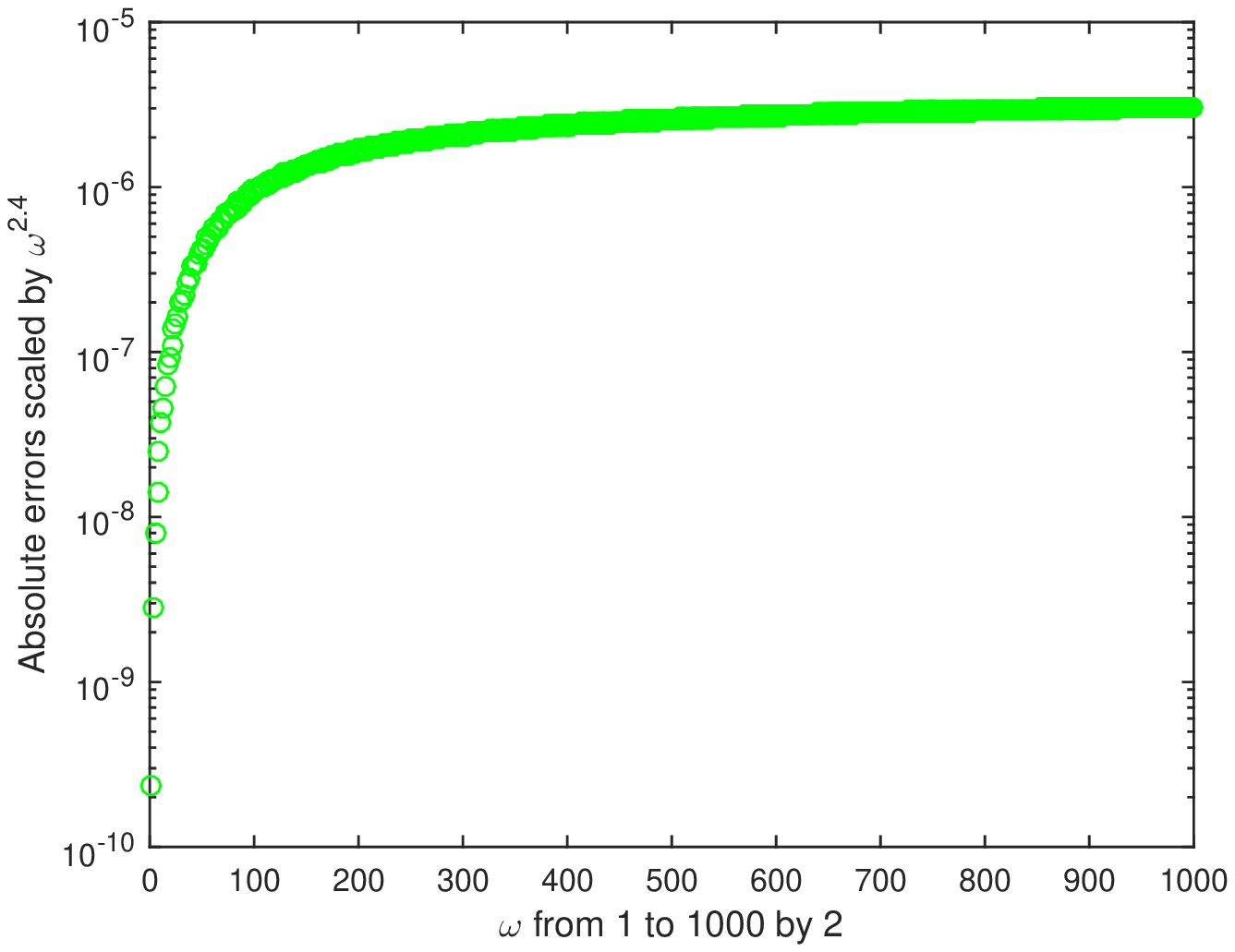}
\caption{Absolute errors scaled by $\omega^{1.4}$ with $s=0$ (left) and $\omega^{2.4}$ with $s=1$ (right) for the Clenshaw--Curtis--Filon--type method for the integral \eqref{eq:d1}
when $N=4$, $k=50$, $\omega$  from $1$ to $1000$ by 2.}\label{fig:5}
\end{figure}

\begin{figure}[hbtp]
 \center
\includegraphics[width=0.46\textwidth, height=0.18\textheight]{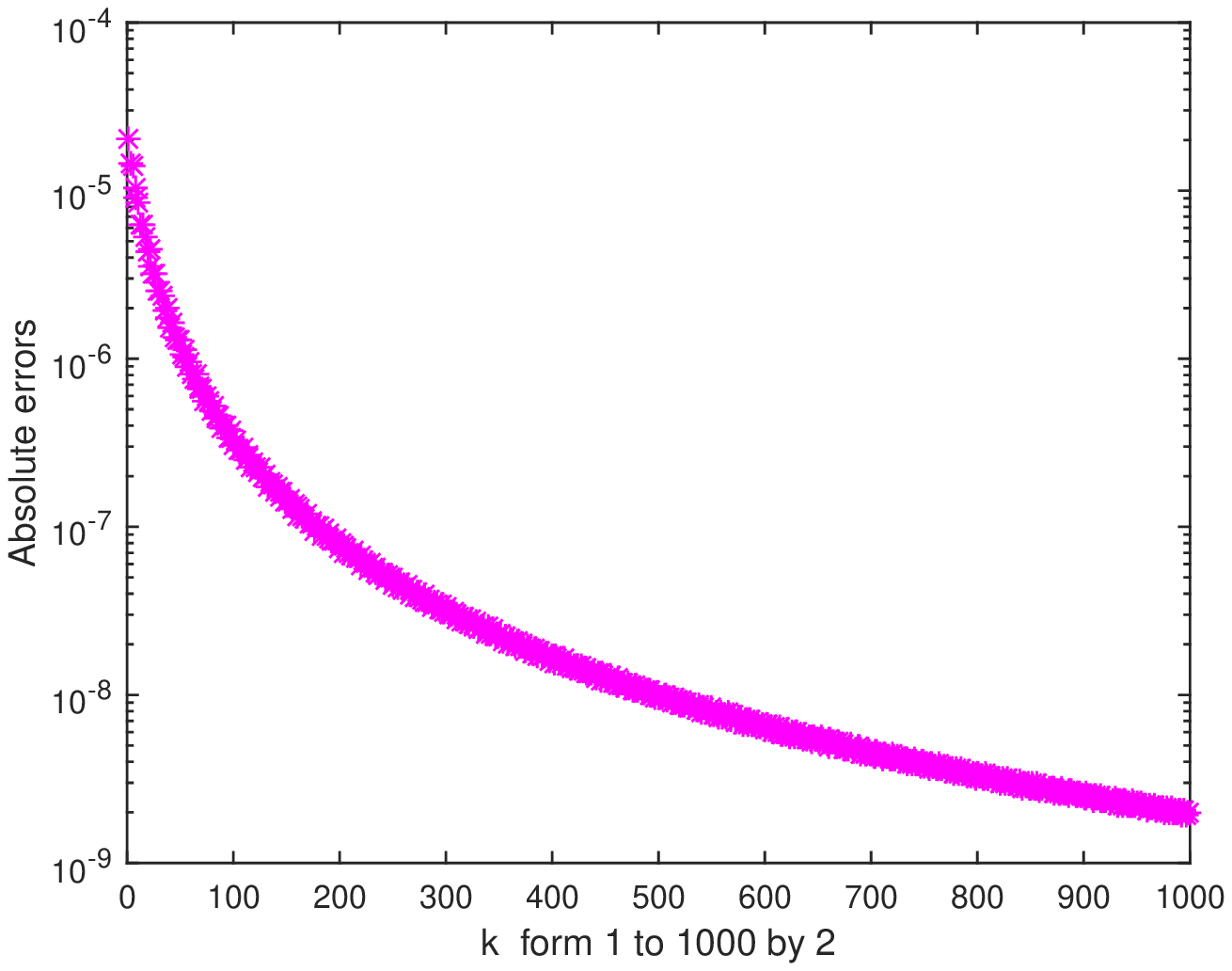}
\includegraphics[width=0.46\textwidth, height=0.18\textheight]{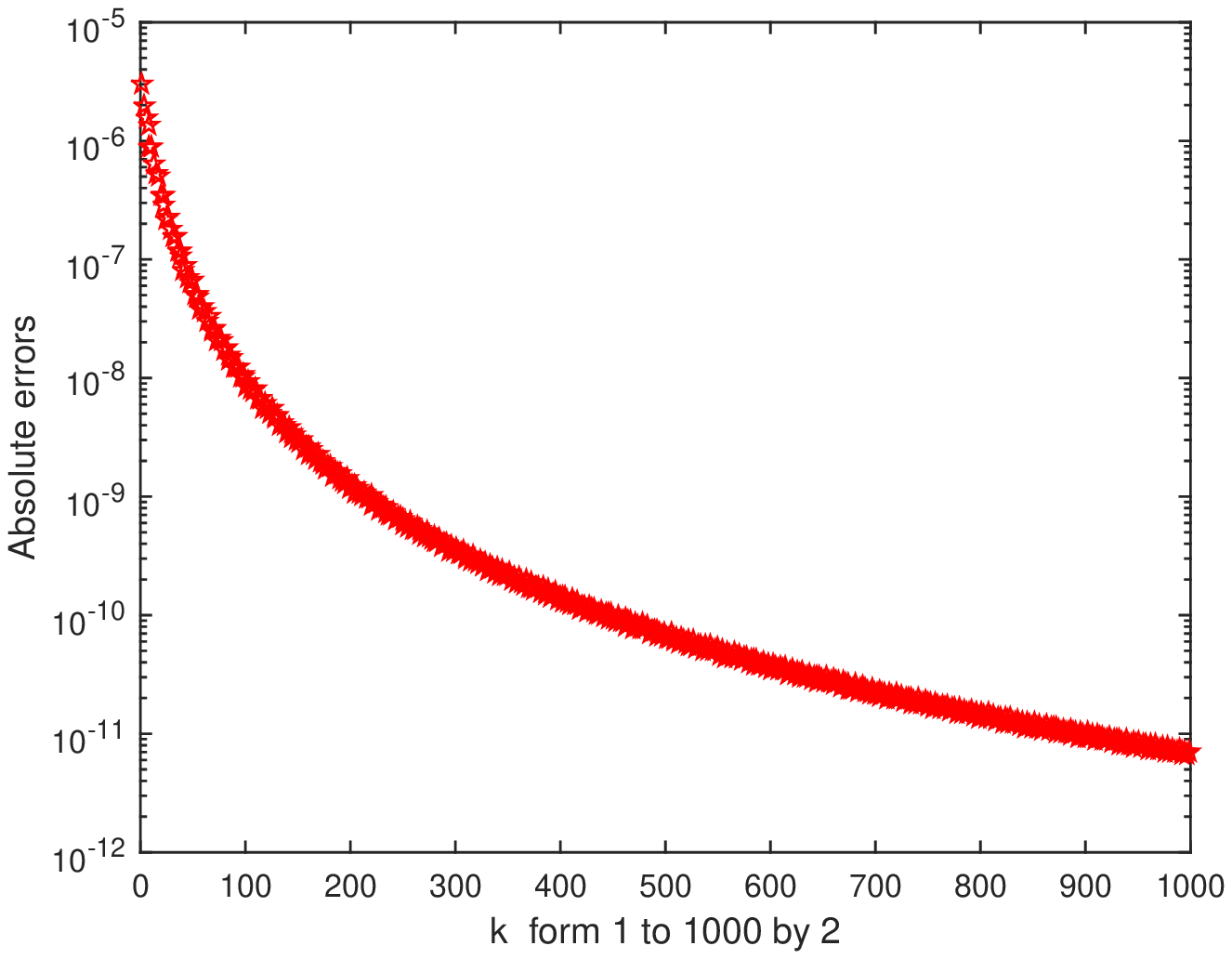}
\caption{Absolute errors for the Clenshaw--Curtis--Filon--type method for the integral \eqref{eq:d2} with $s=1$ (left) and $s=2$ (right), when
$N=6$, $\omega=50$, $k$  from $1$ to $1000$ by 2.}\label{fig:6}
\end{figure}
\begin{figure}[hbtp]
 \center
\includegraphics[width=0.46\textwidth, height=0.18\textheight]{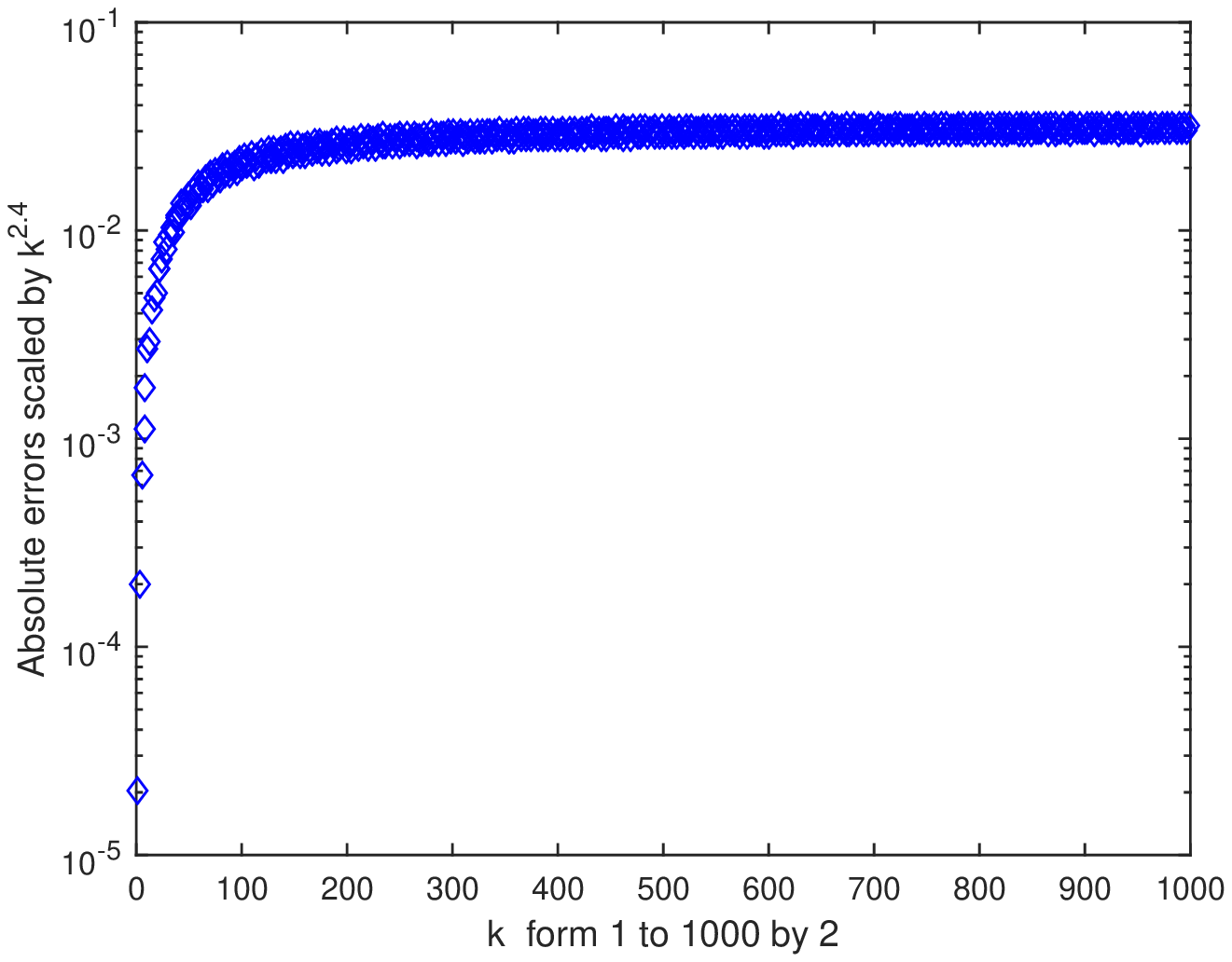}
\includegraphics[width=0.46\textwidth, height=0.18\textheight]{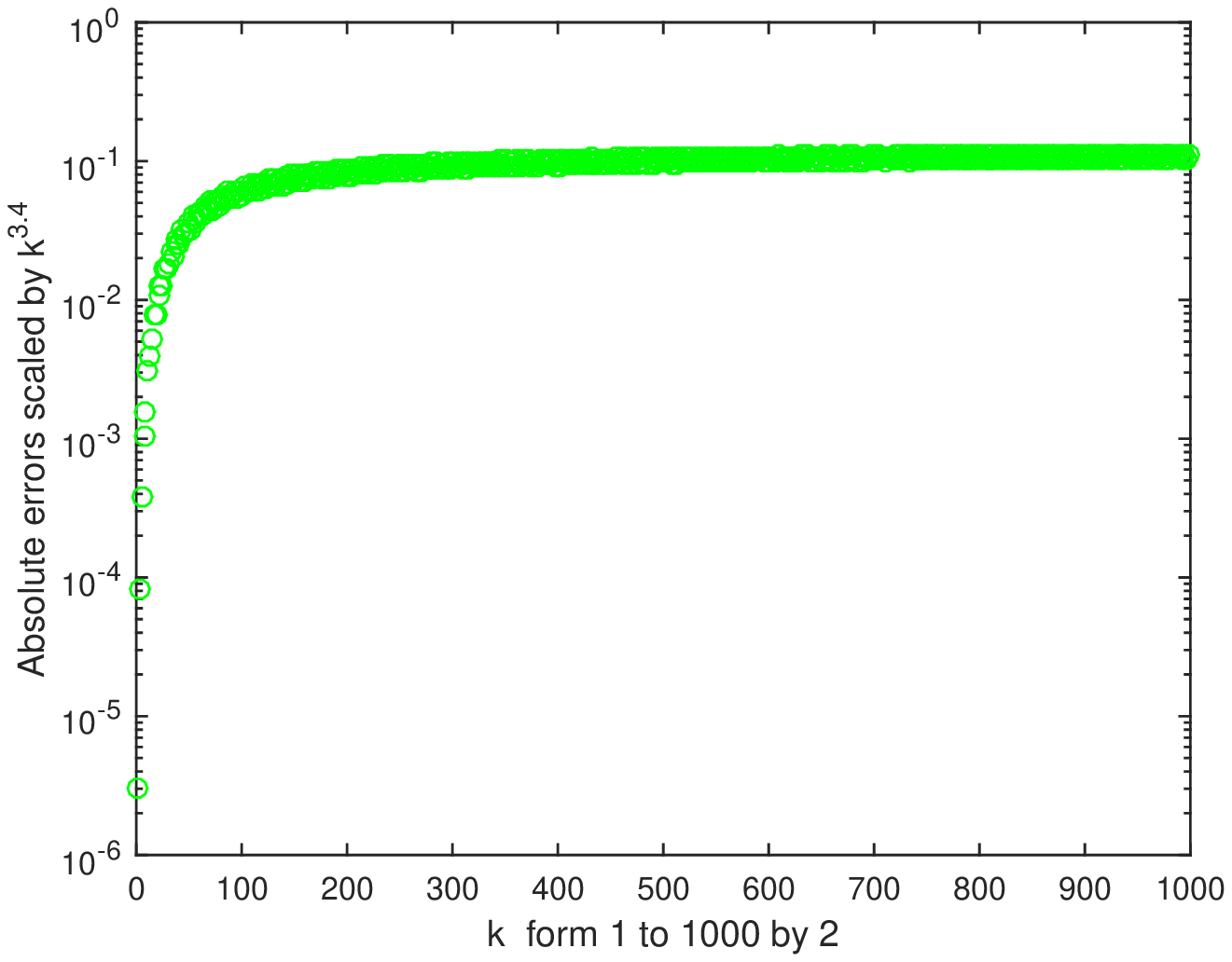}
\caption{Absolute errors scaled by $k^{2.4}$ with $s=1$ (left) and $k^{3.4}$ with $s=2$ (right) for the Clenshaw--Curtis--Filon--type method for the integral \eqref{eq:d2}
when $N=6$, $\omega=50$, $k$  from $1$ to $1000$ by 2.}\label{fig:7}
\end{figure}

\begin{figure}[hbtp]
 \center
\includegraphics[width=0.46\textwidth, height=0.18\textheight]{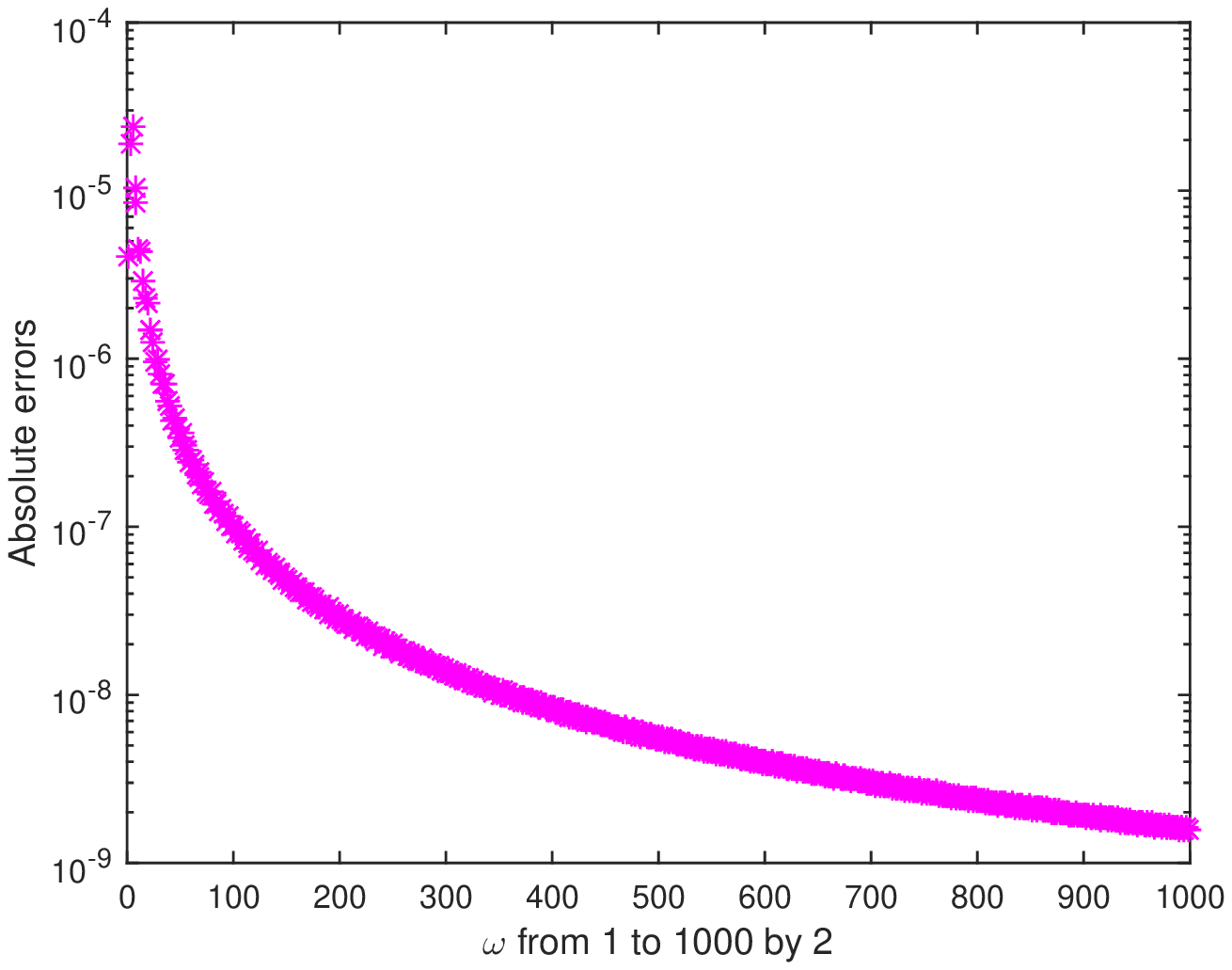}
\includegraphics[width=0.46\textwidth, height=0.18\textheight]{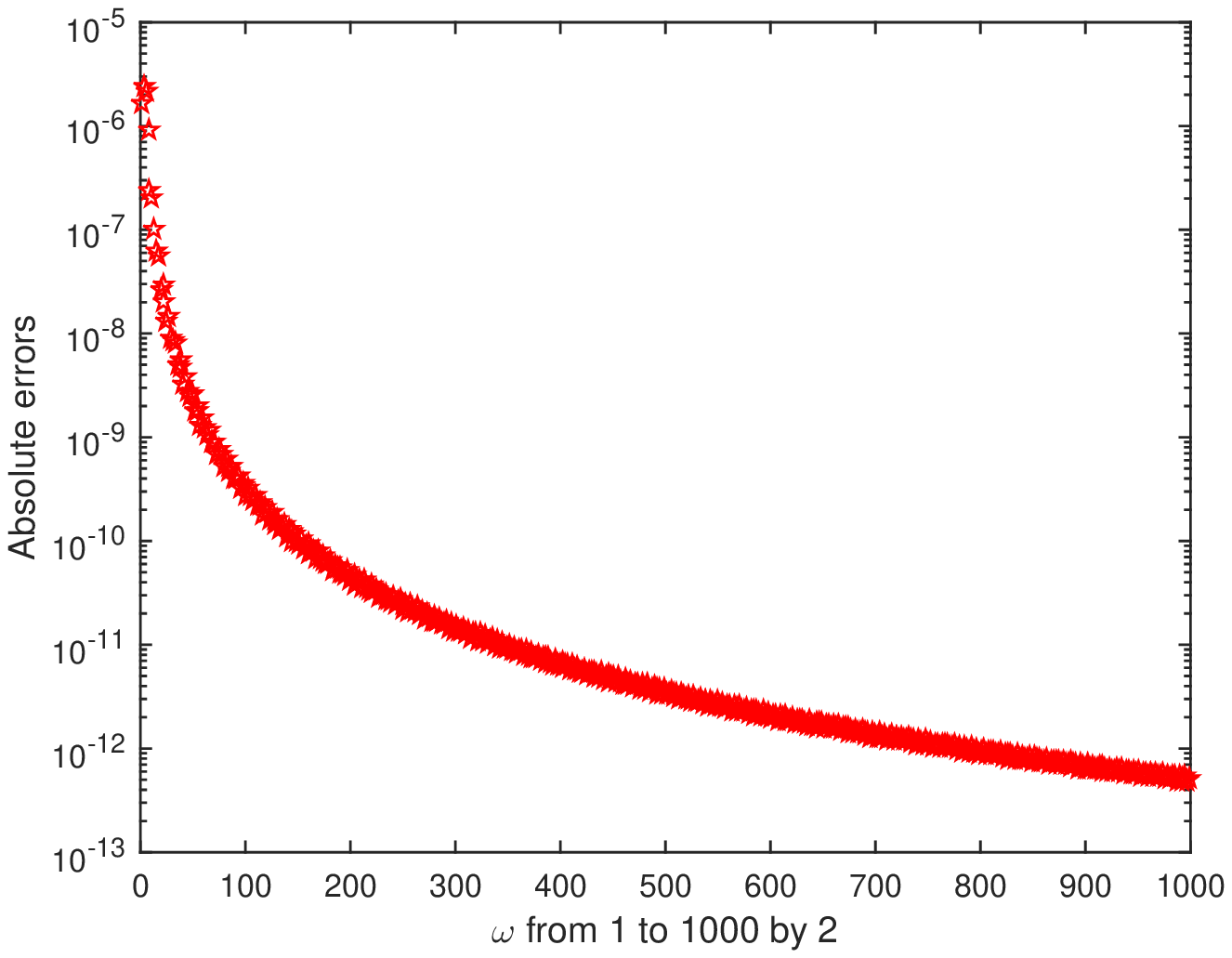}
\caption{Absolute errors for the Clenshaw--Curtis--Filon--type method for the integral \eqref{eq:d3} with $s=0$ (left) and $s=1$ (right), when
$N=4$, $\omega$  from $1$ to $1000$ by 2.}\label{fig:8}
\end{figure}
\begin{figure}[hbtp]
 \center
\includegraphics[width=0.46\textwidth, height=0.18\textheight]{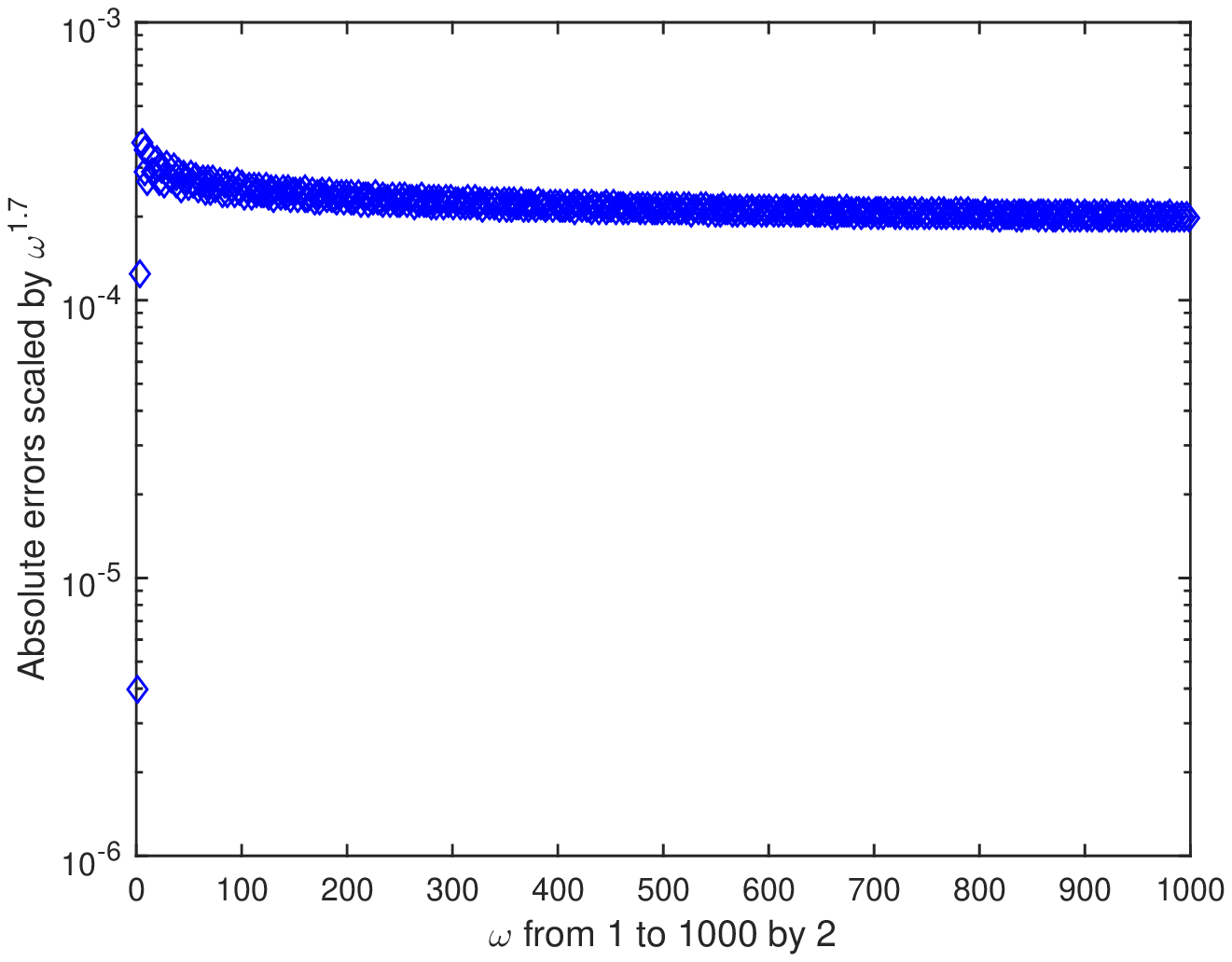}
\includegraphics[width=0.46\textwidth, height=0.18\textheight]{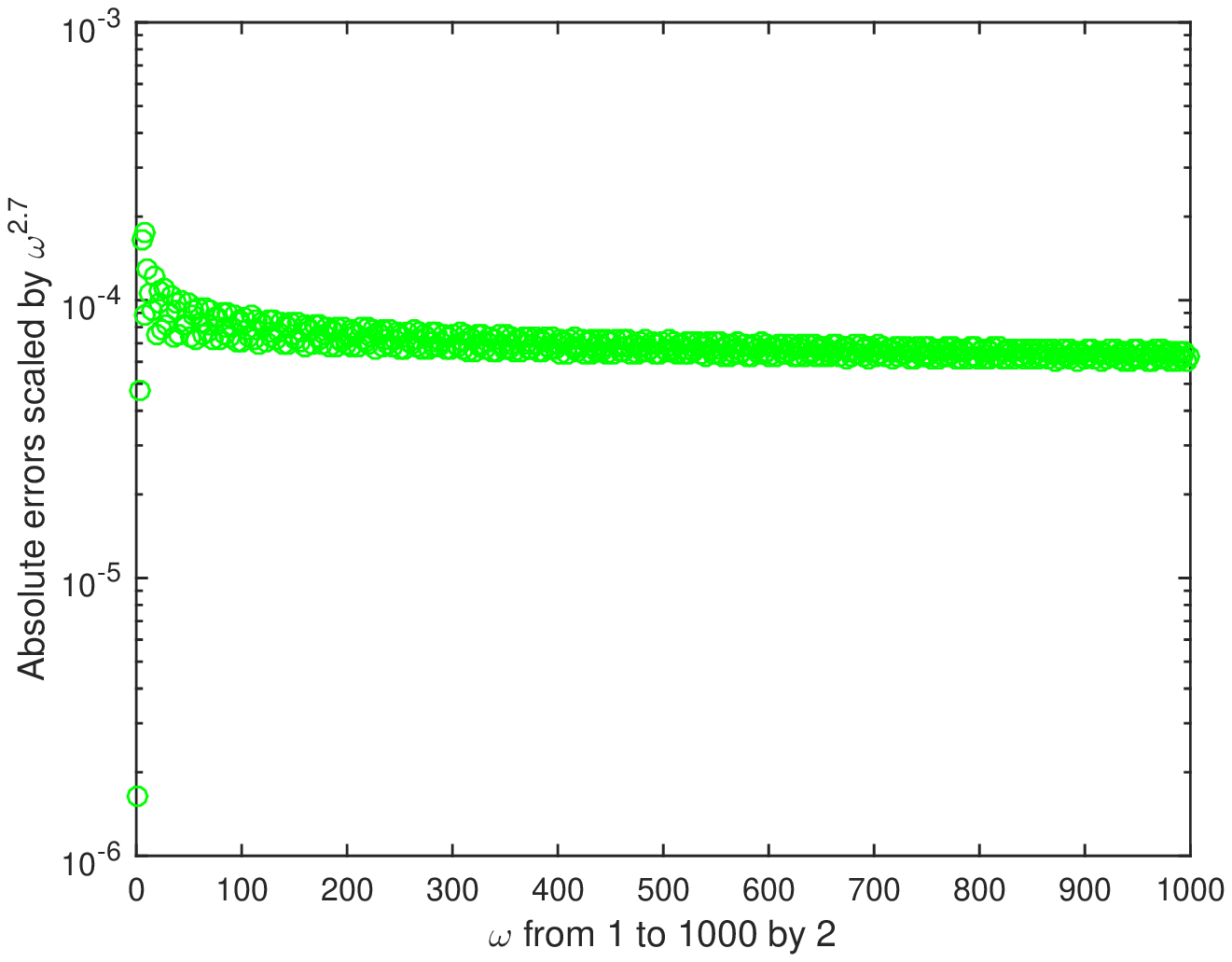}
\caption{Absolute errors scaled by $\omega^{1.7}$ with $s=0$ (left) and $\omega^{2.7}$ with $s=1$ (right) for the Clenshaw--Curtis--Filon--type method for the integral \eqref{eq:d3}
when $N=4$, $\omega$  from $1$ to $1000$ by 2.}\label{fig:9}
\end{figure}

\begin{table}[!htb]
\centering \caption{ Relative errors  for the integral (\ref{eq:d1})
by the  Clenshaw--Curtis--Filon--type method with $k=10$, $N=2,4,6$ and $s=0,1,2$.}\label{tb1}
\begin{tabular}{ccccc}
\toprule
$s$ & $N$ & $\omega=10$ &  $\omega=20$ &$\omega=50$\\
\midrule
$0$ & $2$ & $ 1.78\times 10^{-4}$ &$1.35\times 10^{-4} $&$7.60\times 10^{-5} $\\
$$ & $4$ & $ 1.35\times 10^{-6}$ &$8.93\times 10^{-7} $&$5.22\times 10^{-7} $\\
$$ & $6$ & $ 3.34\times 10^{-9}$ &$1.97\times 10^{-9} $&$1.20\times 10^{-9} $\\
\noalign{\smallskip}
$1$ & $2$ & $ 3.94\times 10^{-7}$ &$ 1.96\times 10^{-7}$&$5.32\times 10^{-8} $\\
$$ & $4$ & $ 1.04\times 10^{-9}$ &$ 6.75\times 10^{-10}$&$1.71\times 10^{-10} $\\
$$ & $6$ & $ 1.72\times 10^{-12} $&$9.28\times 10^{-13} $&$2.49\times 10^{-13} $\\
\noalign{\smallskip}
$2$ & $2$ &  $6.56\times 10^{-10}$ &$ 2.20\times 10^{-10}$&$4.47\times 10^{-11} $\\
$$ & $4$ & $  1.48\times 10^{-12}$ &$ 3.74\times 10^{-13}$&$7.76\times 10^{-14} $\\
$$ & $6$ & $ 1.89\times 10^{-15}$ &$6.79\times 10^{-16} $&$1.26\times 10^{-16}$\\
\noalign{\smallskip}
 \multirow{2}{*}{Real Values}  &$$& $ 0.841824877078759$ &$0.708386698058846 $&$0.517419675175559 $\\
&   $$         & $-1.172097304662626\,\I$ &$-0.956797421788702\,\I$&$-0.711685588704216\,\I$\\
\bottomrule
\end{tabular}
\end{table}

\begin{table}[!htb]
\centering \caption{Relative errors  for the integral (\ref{eq:d2})
by the  Clenshaw--Curtis--Filon--type method with $\omega=10$, $N=8, 16, 24$ and $s=0,1,2$.}\label{tb2}
\begin{tabular}{ccccc}
\toprule
$s$ & $N$ & $k=80$ &  $k=160$ &$k=320$\\
\midrule
$0$ & $8$ & $ 4.36\times 10^{-4}$ &$2.19\times 10^{-4} $&$1.11\times 10^{-4} $\\
$$ & $16$ & $ 1.51\times 10^{-6}$ &$8.45\times 10^{-7} $&$4.13\times 10^{-7} $\\
$$ & $24$ & $ 3.11\times 10^{-9}$ &$1.12\times 10^{-9} $&$3.53\times 10^{-10} $\\
\noalign{\smallskip}
$1$ & $8$ & $ 4.80\times 10^{-6}$ &$ 8.59\times 10^{-7}$&$2.62\times 10^{-7} $\\
$$ & $16$ & $ 7.80\times 10^{-8}$ &$ 1.48\times 10^{-8}$&$3.37\times 10^{-9} $\\
$$ & $24$ & $ 7.96\times 10^{-10} $&$1.28\times 10^{-10} $&$2.61\times 10^{-11} $\\
\noalign{\smallskip}
$2$ & $8$ &  $ 7.96\times 10^{-7}$ &$ 8.89\times 10^{-8}$&$1.19\times 10^{-8} $\\
$$ & $16$ & $  8.10\times 10^{-9}$ &$ 7.17\times 10^{-10}$&$7.77\times 10^{-11} $\\
$$ & $24$ & $  2.95\times 10^{-11}$ &$1.81\times 10^{-12} $&$1.37\times 10^{-13}$\\
\noalign{\smallskip}
 \multirow{2}{*}{Real Values}  &$$& $0.030083151162300$ &$ 0.023581342870858$&$0.017909179561849$\\
&   $$         & $-0.042241981991079\,\I$ &$-0.031875514971454\,\I$&$-0.024353985798652\,\I$\\
\bottomrule
\end{tabular}
\end{table}

\begin{table}[!htb]
\centering \caption{Relative errors  for the integral (\ref{eq:d3})
by the  Clenshaw--Curtis--Filon--type method with  $N=3, 6, 9$ and $s=0,1,2$.}\label{tb3}
\begin{tabular}{ccccc}
\toprule
$s$ & $N$ & $\omega=25$ &  $\omega=50$ &$\omega=100$\\
\midrule
$0$ & $3$ & $ 2.26\times 10^{-5}$ &$9.40\times 10^{-6} $&$4.04\times 10^{-6} $\\
$$ &  $6$ & $ 1.33\times 10^{-6}$ &$5.97\times 10^{-7} $&$2.75\times 10^{-7} $\\
$$ &  $9$ & $ 2.59\times 10^{-9}$ &$1.29\times 10^{-9} $&$6.98\times 10^{-10} $\\
\noalign{\smallskip}
$1$ & $3$ & $2.03\times 10^{-6}$ &$4.66\times 10^{-7}$&$1.11\times 10^{-7} $\\
$$ &  $6$ & $5.78\times 10^{-10}$ &$1.60\times 10^{-10}$&$2.41\times 10^{-11} $\\
$$ &  $9$ & $4.42\times 10^{-11}$&$1.32\times 10^{-11} $&$2.82\times 10^{-12} $\\
\noalign{\smallskip}
$2$ & $3$ &  $ 1.25\times 10^{-8}$ &$ 1.98\times 10^{-9}$&$2.74\times 10^{-10} $\\
$$ &  $6$ & $  1.86\times 10^{-10}$ &$ 2.26\times 10^{-11}$&$2.37\times 10^{-12} $\\
$$ &  $9$ & $  2.11\times 10^{-13}$ &$7.48\times 10^{-15} $&$2.98\times 10^{-15}$\\
\noalign{\smallskip}
 \multirow{2}{*}{Real Values}  &$$& $0.030229145167903$ &$0.017639904837672$&$0.010310330002264$\\
&   $$                      & $-0.034246416918332\,\I$ &$-0.019163197919570\,\I$&$-0.010688289764988\,\I$\\
\bottomrule
\end{tabular}
\end{table}

\textbf{Example 4.1.}
Let us consider the computation of the integral
\begin{eqnarray}\label{eq:d1}
\int_0^1 x^\alpha (1-x)^\beta \cos(x) \E^{\I2kx} H_\nu^{(1)}(\omega x)\D x
\end{eqnarray}
by the Clenshaw--Curtis--Filon--type method \eqref{eq:b3}, where $\nu=0, \alpha=-0.6,\beta=-0.3$.
The absolute errors and scaled absolute errors are displayed in Figs. \ref{fig:4}--\ref{fig:5}, respectively.
Also, the relative errors are displayed in Table \ref{tb1}.

\textbf{Example 4.2.}
Let us consider the computation of the integral
\begin{eqnarray}\label{eq:d2}
\int_0^1 x^\alpha (1-x)^\beta \frac{1}{1+16x^2} \E^{\I 2kx} H_\nu^{(1)}(\omega x)\D x
\end{eqnarray}
by the Clenshaw--Curtis--Filon--type method \eqref{eq:b3}, where $\nu=0.6, \alpha=0,\beta=-0.3$ (Figs. \ref{fig:6}--\ref{fig:7}, Table \ref{tb2}).

\textbf{Example 4.3.}
Finally, we consider the computation of the integral of a special form
\begin{eqnarray}\label{eq:d3}
\int_0^1 x^\alpha (1-x)^\beta \frac{1}{1+(1+x)^2} \E^{\I \omega x} H_\nu^{(1)}(\omega x)\D x
\end{eqnarray}
by the Clenshaw--Curtis--Filon--type method \eqref{eq:b3}, where $\nu=0.3$, and
$\alpha=-0.2,\beta=-0.3$. Figs. \ref{fig:6}--\ref{fig:7} show  error bound on $\omega$ for the Clenshaw--Curtis--Filon--type method for this case.
Table \ref{tb3} displays the  relative errors for the proposed method with $N=3,6,9$ and $s=0,1,2$.

Form Figs. \ref{fig:5}, \ref{fig:7}, \ref{fig:9}, we can see that the error bounds given in Theorem \ref{th:3} for the
Clenshaw--Curtis--Filon--type method are attainable. Figs. \ref{fig:4}, \ref{fig:6}, \ref{fig:8} and Tables \ref{tb1}--\ref{tb3}
show that the presented method is very efficient for the approximation of the
integral \eqref{eq:a1}.  Moreover, for the well-behaved function $f(x)$, the integral \eqref{eq:a1} can be
efficiently approximated by Clenshaw--Curtis--Filon--type method with a small number of interpolation points.
In addition, the improvement of the accuracy for the  integral \eqref{eq:a1} can be obtained
by using interpolation with derivatives of higher order at two endpoints, or adding the number of the
interpolation points.

\section{\textbf{Concluding remarks}}\label{sec:5}
In this paper, we consider a Clenshaw--Curtis--Filon--type method for
the computation of the integral \eqref{eq:a1} with $(N+1)$  Clenshaw--Curtis points, which can be efficiently implemented
in $O(N\log N)$ operations. Moreover, we present a universal  method  for the derivation of the recurrence relation for the modified moments,
which can be applied to the modified moments with other type kernels. Based on this  recurrence relation,
the modified moments can be efficiently computed by using the special functions or the existing method with small number of points.
Finally, an error bound on $k$ and $\omega$ and several numerical experiments are given
to show the accuracy and efficiency for the proposed method.


%




\begin{thebibliography}{99}

\bibitem{Abram}
M. Abramowitz and I.A. Stegun, {\it Handbook of Mathematical
Functions,} National Bureau of Standards, Washington, D.C., 1964.

\bibitem{Arden}
S. Arden, S.N. Chandler--Wilde and S. Langdon, A collocation method for high-frequency scattering by convex polygons,
{\it J. Comput. Appl. Math.}  204 (2007) 334--343.

\bibitem{Arfken}
G. Arkfen, {\it Mathematical Methods for Physicists,} third ed.,
Academic Press, Orlando, Fl, 1985.

\bibitem{Bao}
G. Bao, W. Sun, A fast algorithm for the electromagnetic scattering form a
large cavity, {\it SIAM J. Sci. Comput.} 27 (2005) 553--574.
%
\bibitem{Bateman}
H. Bateman, A. Erd\'{e}lyi, {\it Higher Transcendental Functions,
Vol. I,} McGraw--Hill, New York, 1953.

\bibitem{Bleistein}
N. Bleistein and R. Handelsman, A generalization of the method of steepest escent,
{\it IMA J. Numer. Anal.} 10 (1972) 211--230.
\bibitem{Chen}
R. Chen, Numerical approximations to integrals with a highly
oscillatory Bessel kernel, {\it Appl. Numer. Math.} 62 (2012)
636--648.

\bibitem{Chen2}
R. Chen, On the evaluation of Bessel transformations with the
oscillators via asymptotic series of Whittaker functions, {\it J.
Comput. Appl. Math.} 250 (2013) 107--121.

\bibitem{Chen4}
R. Chen, C. An, On evaluation of Bessel transforms with oscillatory and
algebraic singular integrands, {\it J.
Comput. Appl. Math.} 264 (2014) 71--81.

\bibitem{Chen3}
R. Chen, Numerical approximations for highly oscillatory Bessel
transforms and applications, {\it J.
Math. Anal. Appl. } 421 (2015) 1635--1650.

\bibitem{Davis}
P. J. Davis and D. B. Duncan, Stability and convergence of collocation
schemes for retarded potential integral equations, {\it SIAM J. Sci. Comput.} 42 (2004) 1167--1188.


\bibitem{Dominguez1}
V. Dom\'{\i}nguez,  I. G. Graham,  V. P. Smyshlyaev, Stability and
error estimates for Filon--Clenshaw--Curtis rules for
highly--oscillatory integrals, {\it IMA J. Numer. Anal.} {31} (2011)
1253--1280.
\bibitem{Dominguez2}
V. Dom\'{\i}nguez, I. G. Graham, T. Kim, Filon--Clenshaw--Curtis
rules for highly-oscillatory integrals with algebraic singularities
and stationary points, {\it SIAM J. Numer. Anal.} {51} (2013)
1542--1566.

\bibitem{Erdelyi}
A. Erd\'{e}lyi, Asymptotic representations of Fourier integrals and
the method of stationary phase, {\it J. Soc. Ind. Appl. Math.} 3
(1955) 17--27.

\bibitem{Evans1}
 G. A. Evans and J. R. Webster, A high order progressive method for the evaluation
 of irregular oscillatory integrals, {\it Appl. Numer. Math.} 23 (1997) 205--218.

\bibitem{Evans2}
 G. A. Evans and K. C. Chung, Some theoretical aspects of generalised quadrature methods, {\it J.
Complex.} 19 (2003) 272--285.

\bibitem{Filon}
L. N. G. Filon, On a quadrature formula for trigonometric integrals,
{\it Proc. Royal. Soc. Edinburgh.} \textbf{49} (1928), 38--47.


\bibitem{He}
G. He, S. Xiang and E. Zhu, Efficient computation of highly oscillatory integrals with weak singularity by Gausstype
method, {\it Int. J. Comput. Math.} (2014). doi: 10.1080/00207160.2014.987761.


\bibitem{besselj}
{\it http://functions.wolfram.com/HypergeometricFunctions/Hypergeometric0F1/26/02/13/0001/.}

\bibitem{bessely}
{\it http://functions.wolfram.com/HypergeometricFunctions/Hypergeometric0F1/26/02/15/0001/.}


\bibitem{Huybrechs1}
D. Huybrechs and S. Vandewalle, On the evaluation of highly
oscillatory integrals by analytic continuation, {\it SIAM J. Numer.
Anal.} 44 (2006), 1026--11048.

\bibitem{Huybrechs2}
D. Huybrechs and S. Vandewalle, A sparse discretisation for integral
equation formulations of high frequency scattering problems, {\it
SIAM J. Sci. Comput.} 29 (2007) 2305--2328.

\bibitem{Iserles}
A. Iserles and S. P. N{\o}rsett, Efficient quadrature of highly
oscillatory integrals using derivatives, {\it Proc. Royal Soc. A.}
\textbf{461} (2005), 1383--1399.

\bibitem{Kang1}
H. Kang and S. Xiang, Efficient quadrature of highly oscillatory integrals with
algebraic singularities, {\it J. Comput. Appl. Math.} {237}
(2013), 576--588.

\bibitem{Kang2}
H. Kang, S. Xiang and G. He, Computation of integrals with oscillatory and
singular integrands using Chebyshev expansions, {\it J. Comput. Appl. Math.} {242}
(2013), 141--156.



\bibitem{Kang3}
H. Kang and X. Shao, Fast computation of singular oscillatory Fourier transforms,
{\it Abstr. Appl. Anal.} (2014) 1--8, art. no. 984834.

\bibitem{Kang4}
H. Kang and C. Ling, Computation of integrals with oscillatory singular factors of
algebraic and logarithmic type,  {\it J. Comput. Appl. Math.} {285}
(2015), 72--85.

\bibitem{Levin1}
D. Levin, Fast integration of rapidly oscillatory functions, {\it J.
Comput. Appl. Math.} 67 (1996) 95--101.

\bibitem{Levin2}
D. Levin, Analysis of a collocation method for integrating rapidly oscillatory functions,
 {\it J. Comput. Appl. Math.} 78 (1997) 131--138.

\bibitem{Lozier}
D. W. Lozier, {\it Numerical solution of linear difference equations}, Report NBSIR 80-1976, National Bureau of Standerds, Washington, D.C., 1980.

\bibitem{Luke}
Y. L. Luke,  {\it The Special Functions and Their Approximations, Vol.
I.,} Academic Press, London, 1969.

 \bibitem{Mason}
J. C. Mason and D. C. Handscomb, {\it Chebyshev Polynomials}, Chapman
and Hall/CRC, New York, 2003.

\bibitem{Meijer}
{\it http://functions.wolfram.com/HypergeometricFunctions/MeijerG/21/02/07/00\\
01/.}

\bibitem{Oliver}
J. Oliver, The numerical solution of linear recurrence relations,
{\it Numer. Math.} 11 (1968) 349--360.


\bibitem{Olver1}
S. Olver, Numerical approximation of vector-valued highly
oscillatory integrals, {\it BIT Numer. Math.} 47 (2007) 637--655.

\bibitem{Oreshkin}
B. Oreshkin, {\it http://www.mathworks.com/matlabcentral/fileexchange/31490-meijerg/content/\\MeijerG/MeijerG.m}.

\bibitem{Piessens}
R. Piessens, M. Branders, On the computation of Fourier transforms of singular functions,
{\it J. Comput. Appl. Math.} 43 (1992) 159--169.

\bibitem{Sloan1}
I. H. Sloan and W. E. Smith, Product-integration with the
Clenshaw--Curtis and related points, {\it Numer. Math.} 30 (1978)
415--428.

\bibitem{Sloan2}
I. H. Sloan and W. E. Smith, Product integration with the
Clenshaw--Curtis points: implementation and error estimates, {\it
Numer. Math.} 34 (1980) 387--401.


\bibitem{Xiang}
S. Xiang, Efficient Filon-type methods for $\int_a^b f(x)\E^{\I \omega
g(x)}\D x$, {\it Numer. Math.} {105} (2007), 633--658.

\bibitem{Xiang1}
S. Xiang and H. Wang, Fast integration of highly oscillatory
integrals with exotic oscillators, {\it Math. Comput.} 79 (2010)
829--844.

\bibitem{Xiang2}
S. Xiang, Y. Cho, H. Wang and H. Brunner, Clenshaw--Curtis--Filon--type
methods for highly oscillatory Bessel transforms and applications,
{\it IMA J. Numer. Anal.} 31 (2011) 1281--1314.


\bibitem{Xu3}
Z. Xu, S. Xiang, Numerical evaluation of a class of highly oscillatory integrals
involving Airy functions, {\it Appl. Math. Comput.} 246 (2014) 54--63.


\bibitem{Xu1}
Z. Xu, G.V. Milovanovi\'c and S. Xiang, Efficient computation of highly oscillatory
integrals with Hankel kernel, {\it Appl. Math. Comput.} 261 (2015) 312--322.


\bibitem{Xu2}
Z. Xu,  S. Xiang, On the evaluation of highly oscillatory finite Hankel transform using special functions, {\it Numer. Algor.} (2015)
doi: 10.1007/s11075-015-0033-3.
\end{thebibliography}
%

\end{document}